\documentclass[10pt]{amsart}
\usepackage{amssymb}
\usepackage{graphicx}
\usepackage{xcolor} % A package to add color.
\usepackage{tensor}
\usepackage{fullpage} % Sets all margins to 1 in.  
\usepackage{amsmath}
\usepackage{amsthm}
\usepackage{verbatim}
\usepackage{hyperref}
\usepackage{grffile}  %added
\usepackage{subfig} %added
\usepackage{tabu} %added
\usepackage{svg}  %added
\DeclareMathAlphabet\mathbfcal{OMS}{cmsy}{b}{n} %added

\usepackage{multirow} %added
\usepackage{array} %added
\newcolumntype{L}{>{$}l<{$}} %added
\newcolumntype{D}{>{$\displaystyle}l<{$}} %added
\newcolumntype{C}{>{$}c<{$}} %added

\usepackage{enumitem}
\setlist[itemize]{leftmargin=1.5em}

\usepackage{seqsplit}
\definecolor{green}{rgb}{0,0.8,0} % Redefines the color green.
\newtheorem{theorem}{Theorem}[section]
\newtheorem{corollary}[theorem]{Corollary}
\newtheorem{lemma}[theorem]{Lemma}
\newtheorem{proposition}[theorem]{Proposition}

\theoremstyle{definition}
\newtheorem{definition}[theorem]{Definition}

\theoremstyle{remark}
\newtheorem{remark}[theorem]{Remark}

\numberwithin{equation}{section}
\newcommand{\nnrm}[1]{{\vert\kern-0.25ex\vert\kern-0.25ex\vert #1 
    \vert\kern-0.25ex\vert\kern-0.25ex\vert}}

\def\d{\,\mathrm{d}}

\newcommand{\bfu}{{\bf u}}
\newcommand{\bfv}{{\bf v}}

\newcommand{\bfx}{{\bf x}}

\newcommand{\bfz}{{\bf z}}

\newcommand{\bbE}{\mathbb E}

\newcommand{\bbN}{\mathbb N}

\newcommand{\bbP}{\mathbb P}

\newcommand{\calE}{\mathcal E}

\newcommand{\calG}{\mathcal G}
\newcommand{\calH}{\mathcal H}

\newcommand{\calJ}{\mathcal J}

\newcommand{\calN}{\mathcal N}

\newcommand{\calV}{\mathcal V}

\newcommand{\R}{\mathbb R}

\newcommand{\lt}{\left}
\newcommand{\rt}{\right}
\newcommand{\bq}{\begin{equation}}
\newcommand{\eq}{\end{equation}}

\newcommand{\lal}{\langle}
\newcommand{\ral}{\rangle}

\begin{document}

\title{Controlled pattern formation of stochastic Cucker--Smale systems with network structures}%: Title of the article

\author{Young-Pil Choi}
\address{Department of Mathematics, Yonsei University, Seoul 03722, Republic of Korea}
\email{ypchoi@yonsei.ac.kr}

\author{Doeun Oh}
\address{Department of Mathematics, Yonsei University, Seoul 03722, Republic of Korea}
\email{doeun.235@gmail.com}

\author{Oliver Tse}
\address{Department of Mathematics and Computer Science, Eindhoven University of Technology, P.O. Box 513, 5600 MB Eindhoven, The Netherlands}
\email{o.t.c.tse@tue.nl}

\date{\today}

%\thanks{}%
%\subjclass{}%
%\keywords{}%

%\date{\today}%
%\dedicatory{}%
%\commby{}%
% ----------------------------------------------------------------

\maketitle

% ----------------------------------------------------------------

\begin{abstract}
We present a new stochastic particle system on networks which describes the flocking behavior and pattern formation. More precisely, we consider Cucker--Smale particles with decentralized formation control and multiplicative noises on symmetric and connected networks. Under suitable assumptions on the initial configurations and the network structure, we establish time-asymptotic stochastic flocking behavior and pattern formation of solutions for the proposed stochastic particle system. Our approach is based on the Lyapunov functional energy estimates, and it does not require any spectral information of the graph associated with the network structure.
\end{abstract}

%\tableofcontents

%%%%%%%%%%%%%%%%%%%%%%%%%%%%%%%%%%%%%%%%%%%%%%%%%%
%
%
%
% \section{Introduction}
%
%
%%%%%%%%%%%%%%%%%%%%%%%%%%%%%%%%%%%%%%%%%%%%%%%%%%

\section{Introduction}

Over the past two decades, the mathematical modelling of collective behavior of autonomous self-organized multi-agents into robust patterns has been extensively studied in many different scientific disciplines such as physics, biology, control theory, and applied mathematics due to its biological and engineering applications \cite{BB68,CFTV10, CS07,Sum10,VZ12,War49}. Among them, our interest is related to the Cucker--Smale model for studying flocking behavior \cite{CS07}, which is described by a system of second-order nonlinear ordinary differential equations. Let $x^i_t\in \R^d$ and $v^i_t \in \R^d$ be the position and velocity of the $i$-th agent at time $t\geq 0$, respectively. Then the Cucker--Smale model reads as
\begin{align} \label{CS} \left\{\qquad
\begin{aligned}
\frac{\d}{\d t} x^i_t &= v^i_t, \quad i=1,\dots,N, \quad t > 0, \\[2mm]
\frac{\d}{\d t} v^i_t &= K\sum_{j=1}^N \psi(| x^j_t -x^i_t|)(v^j_t -v^i_t ).
\end{aligned}\right.
\end{align}
The right hand side of the equations for $v^i_t$ describes the nonlocal interaction between agents in velocity. Here $K > 0$ is the coupling strength and $\psi : \R_+ \to \R_+$ is a communication weight function which is usually assumed to be a non-increasing function so that closer agents have stronger influence than those further away, i.e.\ $\psi$ satisfies $0<\psi(r)\leq \psi(s)$ for all $0\leq s\leq r$. For example, the weight function $\psi(r) = (1 + r^2)^{\beta/2}$ was chosen in \cite{CS07}. For the system \eqref{CS}, depending on the exponent $\beta$, unconditional and conditional flocking estimates showing the velocity alignment behavior are discussed in \cite{CFRT10,CS07,HL09,HT08}.

Various modifications and extensions to the classical Cucker--Smale model have been proposed and investigated. The singular communication weight, $\psi(r) = r^{-\beta}$ with $\beta>0$, is considered in \cite{ACHL12,CCMP17,MMPZ19,Pes14} to avoid the collision between agents. A rescaled weight function is introduced in \cite{MT11} to take into account the relative distance between agents. The presence of time delays in information processing is also discussed in \cite{CH17,CL18,CP19,DHK19,EHS16,HM20}. The system \eqref{CS} with additive or multiplicative noises in the velocity measurements are also dealt with in \cite{AH10,BCC11,CDP18,CS19,HLL09,HJNXZ17,TLY14}. We refer to 
\cite{CCP17,CHL17,MMPZ19} and references therein for recent surveys on the Cucker--Smale models and its variants. 

In a recent work \cite{CKPP19} motivated from the context of 1D vehicle platoons, a decentralized controller inducing desired spatial configurations is introduced. More precisely, the following controller $w^i_t$ is added to the equations for $v^i_t$ in \eqref{CS}:
 $$\left\{\qquad \begin{aligned}
w^1 &= -\phi(|x^1 - x^2 - z^1|^2)(x^1 - x^2 - z^1),\cr
w^i &= \phi(|x^{i-1} - x^i - z^{i-1}|^2)(x^{i-1} - x^i - z^{i-1}) -\phi(|x^i - x^{i+1} - z^i|^2)(x^i - x^{i+1} - z^i),\cr
w^N &= \phi(|x^{N-1} - x^N - z^{N-1}|^2)(x^{N-1} - x^N - z^{N-1}),
\end{aligned}\right.$$
for $i =2,\dots, N-1$. Here $z^i \in \R^d$, $i=1,\dots, n-1$ are preassigned values which determines the relative positioning of the agents based on the desired spatial configuration. $\phi(r)$ is a communication weight function that is taken in \cite{CKPP19} to be of the form
\[
\phi(r) = \frac{1}{(1 + r)^\beta}, \quad \beta > 0.
\]
For that system, the time-asymptotic spatial pattern formation showing 
\[
\lim_{t \to \infty} x^i_t = x^i_\infty \quad \mbox{satisfying} \quad x^i_\infty = x^{i-1}_\infty - z^{i-1}, \quad i=2,\dots,N
\]
is proved in \cite{CKPP19} under certain assumptions on the initial data.

The main purpose of the current work is to present a new stochastic Cucker--Smale flocking model with formation control on networks, and to study the time-asymptotic stochastic flocking behavior and pattern formation. We consider a symmetric and connected network that can be realized with an undirected graphs $\calG = (\calV, \calE)$ consisting of a finite set $\calV = \{1,\dots, k\}$ of vertices and a set $\calE \subset \calV \times \calV$ of arcs. $\calN^i:= \{j \in \calV: (i,j) \in \calE\}$ denotes a generic neighbor set of vertex $i$ and may depend on the interaction mechanisms involved\footnote{We exclude self-loops in the graph, i.e.\ $(i,i) \notin \calE$.}. Further, let $(B_t)_{t\geq 0}$ represent the standard real-valued Brownian motion. Then our main stochastic model is given by
\begin{align} \label{main_eq}\left\{\qquad
\begin{aligned}
\d x^i_t &= v^i_t    \d t, \quad i=1,\dots,N, \quad t > 0, \\[2mm]
\d v^i_t &= K\sum_{j\in \calN_\psi^i} \psi(| x^j_t -x^i_t|)(v^j_t -v^i_t )\d t + M u^i_t \d t +\sigma\sum_{j\in \calN_B^i }(v^j_t -v^i_t )\d B_t.
\end{aligned}\right.
\end{align}
The first term on the right-hand side of the equations for $v^i_t$ plays the same role as in the original Cucker--Smale model \eqref{CS}, but now $i$-th agent only interacts with its neighbors $j\in \calN_\psi^i$. The second term represents an external control signal $\bfu_t = (u^1_t,\dots,u^N_t) \in \R^{dN}$ given by
\[
u^i_t = \sum_{j\in \calN_{\phi}^i } \phi (| x^j_t - z^j -  (x^i_t - z^i) |^{2})  (x^j_t - z^j -  (x^i_t - z^i) ), \quad i=1,\dots,N,
\]
where $M>0$ is the strength of the control force and  $\phi:[0,\infty) \to [0,\infty)$ is a regular communication weight function. The vector $\bfz = (z^1,\dots,z^N) \in \R^{dN}$ determines the desired spatial configuration. The third term is weighted multiplicative noise where $\sigma > 0$ is the strength of the noise. This term can be derived from the consideration that $\psi$ contains Gaussian white noise, see \cite{AH10, HJNXZ17} for more details. 

Throughout this paper, we will assume that the communication weight functions 
\begin{enumerate}[label=(\textsf{A})]
	\item $\psi$, $\phi$ are both bounded and Lipschitz continuous with $\psi^{min} > 0$.
\end{enumerate}

In this model, the network structure for all the terms may differ, which is a generalization of known models. Here we assume that the graphs $\calG_\zeta = (\calV_\zeta, \calE_\zeta)$ with $\zeta \in \{\psi,\phi,B\}$ are symmetric and connected, i.e.\ the following assumptions hold: 
\begin{enumerate}[label=(\textsf{B\arabic*})]
\item $(i,j)\in \calE_\zeta$ if and only if $(j,i) \in \calE_\zeta$ for all $1 \leq i,j \leq N$,
\item for any $i,j \in \calV_\zeta$ with $i\neq j$, there exists a shortest path from $i$ to $j$, say
\[
i = p_1 \to p_2 \to \cdots \to p_{d_{ij}} = j, \quad (p_k, p_{k+1}) \in \calE_\zeta, \quad k=1,2,\dots,d_{ij} -1.
\]
\end{enumerate}

For notational simplicity, we will use the following simplified notations throughout the paper:
\[
x^{ij}:= x^j -x^i,\quad  \bar x^{ij} :=\bar x^j- \bar x^i , \quad  \bar x^i :=x^i -z^i \quad \mbox{for} \quad i,j=1,\dots,N,
\]
\[
\xi^{max}:= \sup_{r \geq 0} \xi(r), \quad \xi^{min}:= \inf_{r \geq 0} \xi(r) \quad \mbox{for} \quad \xi \in \{\psi, \phi \}.
\]
For any $\bfz=(z^1,\ldots,z^N)\in (\R^{d})^N$, we set
\[
\|\bfz\| := \lt(\sum_{i=1}^N |z^i|^2 \rt)^{1/2}, \quad z^{ave} := \frac1N\sum_{i=1}^N z^i\in \R^d\,.
\]
We further set $\bfz^{ave} := (z^{ave},\ldots,z^{ave})\in (\R^d)^N$ to be the $N$-concatenation of the vector $z^{ave}$.

\medskip

Before stating our main results, we introduce a notion of stochastic flocking for the system \eqref{main_eq}.
\begin{definition}\label{def_flocking} The stochastic particle system \eqref{main_eq} exhibits {\it time-asymptotic stochastic flocking} if and only if the pair $(\bfx_t,\bfv_t)$ satisfies the following two conditions:
\begin{enumerate}[label=(\roman*)]
\item The difference of all pairwise velocities goes to zero in expectation:
\[
\lim_{t \to \infty} \max_{1 \leq i,j \leq N} \bbE |v^i_t - v^j_t|^2 = 0.
\]
\item The relative distances between stochastic particles are uniformly bounded in time $t$ in expectation:
\[
\sup_{0 \leq t < \infty}\max_{1 \leq i,j \leq N} \bbE |x^i_t - x^j_t|^2 < \infty.
\]
\end{enumerate}
\end{definition}

Our first main theorem provides a sufficient condition for the time-asymptotic stochastic flocking behavior of the system \eqref{main_eq} in the sense of the above definition.

\begin{theorem}\label{main_thm} Let $(\bfx_t,\bfv_t)$ be the global pathwise unique solution to the stochastic particle system \eqref{main_eq} with the random initial data $(\bfx_0,\bfv_0)$ satisfying $\bbE[ \|\bfx_0\|^2+\|\bfv_0\|^2]<\infty$. Suppose that 
\begin{enumerate}[label=(\roman*)]
	\item the coupling strength $K > 0$ satisfies 
\bq\label{condi_k}
K > \frac{\sigma^2 c_B}{\psi^{min}}\Bigl(1 + d(\calG_\psi) |\calE_\psi^c |\Bigr)\,,\qquad c_B:=\max_{i\in \calV_B} |\calN_B^i|\,,
\eq
where for any graph $\calG=(\calV,\calE)$, $d(\calG)$ is the diameter of the graph and $\calE^c := \calV\times\calV\setminus \calE$, and
\item the initial data satisfy
\bq\label{asp}
	\int_0^\infty \phi(r) \d r > \frac 2M \bbE\|\bfv_0\|^2 + \sum_{(i,j)\in\calE_\phi } \bbE\int_0^{|\bar x^{ij}_0|^2} \phi (r) \d r\,.
\eq
\end{enumerate}
Then the time-asymptotic stochastic flocking occurs in the sense of Definition \ref{def_flocking}.
\end{theorem}

Condition (i) of Theorem~\ref{main_thm} is fairly common in order to obtain stochastic flocking \cite{AH10,CDP18,CS19,HLL09,TLY14} of \eqref{main_eq} without both control terms and varying network structures, while condition (ii) is a sufficient condition to obtain pattern formation (see also \cite{CKPP19} for the deterministic case with a specific network structure). Therefore, Theorem~\ref{main_thm} generalizes both these works to controlled pattern formation of stochastic systems with varying network structures.

\medskip

Our second theorem is on the stochastic pattern formation of the system \eqref{main_eq} under a strictly positive lower bound assumption on $\phi$. For this purpose, we introduce the following energy functional
\[
	\calH(\bfx,\bfv):= \bbE\|\bar\bfx - \bar \bfx^{ave}\|^2 + \bbE \|\bfv - \bfv^{ave}\|^2\,,
\]
where the first term quantifies the error between the position of particles and the desired spatial configuration for the stochastic system \eqref{main_eq}, while the second is used to deduce the emergence of stochastic flocking.

\begin{theorem}\label{main_thm2} Let $(\bfx_t,\bfv_t)$ be the global pathwise unique solution to the stochastic particle system \eqref{main_eq} with the random initial data $(\bfx_0,\bfv_0)$ satisfying $\bbE[ \|\bfx_0\|^2+\|\bfv_0\|^2]<\infty$.  Suppose that the all assumptions in Theorem \ref{main_thm} hold. Moreover, we assume $\phi^{min} > 0$. Then there exist positive constants $p,q>0$, independent of $t$, such that
\[
	\calH(\bfx_t,\bfv_t) \le p\calH(\bfx_0,\bfv_0)\,e^{-q t}\qquad\text{for all $t\ge 0$\,.}
\]
\end{theorem}

We would like to point out that the pattern formation result in \cite{CKPP19}, where the deterministic case is considered, is obtained without any decay rate of convergence. On the other hand, Theorem~\ref{main_thm2} clearly indicates the exponential decay behavior of the $\calH$. We will see, however, that in some of our numerical experiments, this decay does not happen monotonically in time. For this reason, we will introduce an auxiliary energy functional that is equivalent to $\calH$ (see Lemma~\ref{lem_J}), which does decay monotonically in time. We note that the strategy used in the proof of Theorem \ref{main_thm2} can be directly applied to the deterministic case, thus improving the pattern formation result in \cite{CKPP19}.

\medskip

In the case when $\bbE\, x_0^{ave} = z^{ave}$, we obtain the almost sure convergence of $\bfx_t-\bfv^{ave}_0 t$ to the deterministic limit $\bfz$. More precisely, we obtain the following result:

\begin{corollary}
	If, in addition to the assumptions of Theorem~\ref{main_thm2}, $\bbE\, x_0^{ave} = z^{ave}$ holds, then
	\[
		\bbE \|\bfx_t -  \bfv^{ave}_0 t - \bfz\|^2 \to 0
	\]
	exponentially fast as $t \to \infty$. In particular, $\bfx_t - \bfv^{ave}_0 t\to \bfz$ almost surely as $t\to\infty$. 
\end{corollary}
\begin{proof} 
The proof follows by observing that (cf.\ \eqref{ave_x})
\[
\bar x^{ave}_t = x^{ave}_t - z^{ave} = x^{ave}_0 - z^{ave}.
\]
Thus if $ \bar x^{ave}_0 = 0$, then $\bar x^{ave}_t = 0$ for almost every $t \geq 0$, and we conclude with Theorem~\ref{main_thm2}. Moreover, by Chebyshev's inequality, we find
\begin{align*}
	\bbP(\{\omega:\|\bfx_{t_n} -  \bfv^{ave}_0 t_n - \bfz\|\ge \varepsilon\}) \le \frac{1}{\varepsilon^2} \bbE \|\bfx_{t_n} -  \bfv^{ave}_0 t_n - \bfz\|^2 \le \frac{p}{\varepsilon^2} \calH(\bfx_0,\bfv_0)\,e^{-q t_n}
\end{align*}
for any $\varepsilon > 0$ and $t_n \in \R_+$ with $n \in \bbN$, and the right hand side of the above is integrable. Hence, by Borel-Cantelli lemma, we have the almost surely convergence of $\bfx_t - \bfv^{ave}_0$ towards $\bfz$ as $t \to \infty$.
\end{proof}
\begin{remark}
\begin{enumerate}
\item Note that assumption {\it (ii)} of Theorem~\ref{main_thm} is a smallness assumption on the initial data whenever the left-hand side is finite. Clearly, if the weight function $\phi$ is not integrable on $(0,\infty)$, e.g.\ $\phi(r) = (1 + r)^{-\beta}$ with $\beta \leq 1$, then \eqref{asp} is satisfied for any initial data with $\bbE[ \|\bfx_0\|^2+\|\bfv_0\|^2]<\infty$.
		\item As mentioned above, the multiplicative noise can be obtained from the velocity alignment force. This naturally induces the same graphs for the velocity alignment forces and the noises, i.e.\ $\calG_\psi = \calG_B$. In this case, the condition on the coupling strength $K$ \eqref{condi_k} can be replaced with
\[
K > \frac{\sigma^2 c_B}{\psi^{min}}\,.
\]
\end{enumerate}
\end{remark}
 
\subsubsection*{Organization of the paper} The rest of this paper is organized as follows. In Section \ref{sec:pre}, we present the existence and uniqueness of pathwise global solutions to the stochastic particle system \eqref{main_eq}. We also estimate the total energy for the system \eqref{main_eq} which will be used for the stochastic flocking estimates. Section \ref{sec:flock} is devoted to provide the details of proofs for Theorems \ref{main_thm} and \ref{main_thm2}. Finally, in Section \ref{sec:num}, we show numerical simulations validating our theoretical results, and provide further insights into the stochastic flocking, pattern formation, and oscillatory behavior depending on the network structures.

%%%%%%%%%%%%%%%%%%%%%%%%%%%%%%%%%%%%%%%%%%%%%%%%%%
%
%
%
% \section{Preliminaries}
%
%
%%%%%%%%%%%%%%%%%%%%%%%%%%%%%%%%%%%%%%%%%%%%%%%%%%

\section{Preliminaries}\label{sec:pre}

\subsection{Existence and uniqueness of pathwise global solutions}

Our first preliminary result pertains the well-posedness of our stochastic system \eqref{main_eq}. Its proof is rather standard, and makes use of the fact that the drifts and diffusion coefficients are locally Lipschitz continuous, and satisfy linear growth assumptions due to assumption (\textsf{A}) (cf.\ \cite[Theorem~3.1]{D96}). 

\begin{proposition}\label{prop:well-posedness} For any given $T>0$ and random initial data $(\bfx_0, \bfv_0)$ with $\bbE[\|\bfx_0\|^2 + \|\bfv_0\|^2] <\infty$, the stochastic particle system \eqref{main_eq} has a pathwise unique solution $(\bfx_t,\bfv_t)$ on the interval $[0,T]$ satisfying
\[
	\sup_{t\in[0,T]}\bbE\bigl[ \|\bfx_t\|^2 + \|\bfv_t\|^2\bigr] <\infty\,.
\]
\end{proposition}

\subsection{Energy estimates}
In this part, we present the energy estimates for the system \eqref{main_eq}. 

We begin by noticing that the average velocity $v_t^{ave}$ remains constant in time, i.e.\ $v_t^{ave}= v_0^{ave}$ for almost every $t\ge 0$. This follows directly from the assumed symmetry of the networks (\textsf{B1}) and weight functions. 

For the rest of this paper, we may assume, without loss of generality, that the particle velocity average is zero, i.e.\ $v_t^{ave} =0$ for every $t \geq 0$. If necessary, we may consider new variables $(\hat x^i_t, \hat v^i_t) := (x^i_t  - v_0^{ave} t, v^i_t - v_0^{ave})$ since the velocity average is conserved in time. In particular, this also implies
\bq\label{ave_x}
x^{ave}_t := \frac1N\sum_{i=1}^N x^i_t = \frac1N\sum_{i=1}^N x^i_0 + \frac1N\sum_{i=1}^N \int_0^t v^i_s\,ds = \frac1N\sum_{i=1}^N x^i_0 =: x^{ave}_0
\eq
almost surely and for every $t\geq 0$.

\medskip

We also recall from \cite[Lemma A.1]{CLHXY14} the following lemma showing that the position variance can be controlled by the sum of the relative distance of the pair of connected nodes.

\begin{lemma}\label{lem_conn} Suppose that the graph $\calG = (\calV, \calE)$ is connected and let $x^i$ be the position of the $i$-th particle. Then for any $(x^1,\cdots, x^N)$, we have
\[
L_\calG \sum_{i,j=1}^N |x^i - x^j|^2 \leq \sum_{(i,j) \in \calE} |x^i - x^j|^2,
\]
where $L_\calG > 0$ is given by
\[
L_\calG = \frac{1}{1 + d(\calG) |\calE^c |}. 
\]
Here $d(\calG)$ is the diameter of the graph, i.e.\ $d(\calG) = \max_{(i,j) \in \calE} d_{ij}$, and $\calE^c := \calV\times\calV\setminus \calE$.
\end{lemma}

We are now ready to present the total energy estimate. In contrast to the works \cite{AH10,CDP18,CS19,HLL09,TLY14} on stochastic flocking, the presence of controls prevents us from easily obtaining an exponential decay estimate for the kinetic energy in expectation $\bbE\|\bfv_t\|^2$ under the strictly positive lower bound assumption on $\psi$. Nevertheless, we have a uniform-in-time estimate for the total energy in expectation together with the time-integrability of the kinetic energy, which plays a crucial role in establishing the stochastic flocking behavior.

\begin{lemma}\label{lem_energy}
Let $(\bfx_t,\bfv_t)$ be the pathwise unique solution to the stochastic particle system \eqref{main_eq} with the initial data $(\bfx_0,\bfv_0)$. Furthermore, let the assumptions (i)--(ii) of Theorem~\ref{main_thm} be satisfied with $$\lambda:=\psi^{min} L_{\calG_\psi} K - \sigma^2 c_B>0\,.$$ Then for almost every $t\geq 0$,
\begin{align}\label{est_ev}
	\bbE \lt[ \|\bfv_t\|^2+\frac{M}{2} \sum_{(i,j)\in\calE_\phi}\Phi(|\bar x^{ij}_t|^2) \rt] + 2N\lambda\, \int_0^t \bbE\|\bfv_s\|^2\d s \le \bbE\lt[ \|\bfv_0\|^2+\frac{M}{2} \sum_{(i,j)\in\calE_\phi}\Phi(|\bar x^{ij}_0|^2) \rt],
\end{align}
where $\Phi(r):= \int_0^r \phi(s)\d s$. In particular, the map $(0,\infty)\ni t\mapsto\bbE\|\bfv_t\|^2$ is Lipschitz, and the following identity holds for almost every $t\ge 0$:
\begin{align}\label{iden_ev}
\begin{aligned}
	\frac{\d}{\d t}\bbE\|\bfv_t\|^2 &= -K\sum_{(i,j)\in\calE_\psi}\bbE\Bigl[\psi(|x^{ij}_t|) | v^i_t -v^j_t|^2\Bigr] + \sigma^2\sum_{i=1}^N \bbE\lt|\sum_{j\in \calN_B^i }(v^j_t-v^i_t )\rt|^2 \\
	&\hspace{1em} - M \sum_{(i,j) \in \calE_\phi} \bbE\Bigl[\phi(|\bar x^{ij}_t|^2) \lal \bar x^j_t - \bar x^i_t, v^j_t - v^i_t \ral \Bigr]\,.
\end{aligned}
\end{align}
\end{lemma}
\begin{proof}
	An application of It\^o's formula yields the identity
	\begin{align}\label{energy_eq}
\begin{aligned}
\d\|\bfv_t\|^2 &= -K \sum_{(i,j) \in \calE_\psi }\psi(|x^{ij}_t|)|v^j_t -v^i_t|^{2}\d t + \sigma^2\sum_{i=1}^N \left|\sum_{j\in \calN_B^i }(v^j_t-v^i_t )\right|^2\d t \\
&\hspace{1em}- M \sum_{(i,j) \in \calE_\phi} \phi(|\bar x^{ij}_t|^2) \lal \bar x^j_t - \bar x^i_t, v^j_t - v^i_t \ral \d t - \sigma\sum_{(i,j)\in\calE_B } | v^j_t -v^i_t |^{2} \,\d B_t.
\end{aligned}
\end{align}
To deduce estimate \eqref{est_ev}, we observe that
\[%\bq\label{est_x}
\d |\bar x^{ij}_t|^2 = 2\lal \bar x^{ij}_t, \d \bar x^{ij}_t\ral = 2\lal \bar x^{ij}_t,  v^{ij}_t\ral \d t = 2\lal \bar x^j_t -\bar x^i_t, v^j_t - v^i_t\ral \d t,
\]%\eq
which allows us to rewrite the third term on the right-hand side of \eqref{iden_ev} as
\[
M\sum_{(i,j)\in\calE_\phi}\phi(|\bar x^{ij}_t|^2) \lal \bar x^j_t -\bar x^i_t, v^j_t - v^i_t\ral \d t = \d \lt(\frac{M}{2}  \sum_{(i,j)\in\calE_\phi}\int_{0}^{|\bar x^{ij}_t|^2} \phi( r) \d r \rt).
\]
For the first term, we use the lower bound assumption on $\psi$ to deduce
$$\begin{aligned}
-\sum_{(i,j)\in\calE_\psi}\psi(|x^{ij}_t|) | v^i_t - v^j_t|^2 &\leq -\psi^{min}\sum_{(i,j)\in\calE_\psi}  | v^i_t - v^j_t|^2 \leq -\psi^{min}L_{\calG_\psi}\sum_{i,j=1}^N | v^i_t -v^j_t|^2\,,
\end{aligned}$$
where Lemma \ref{lem_conn} was used in the second inequality. As for the second term, we simply estimate as follows:
\[
	\sigma^2\sum_{i=1}^N \lt|\sum_{j\in \calN_B^i }(v^j_t - v^i_t )\rt|^2\leq \sigma^2\sum_{i=1}^N|\calN_B^i|\sum_{j\in\calN_B^i}|v^j_t - v^i_t |^2 \le \sigma^2 c_B \sum_{i,j=1}^N |v^j_t - v^i_t |^2.
\]
Combining the inequalities yields
\[
-K\sum_{(i,j)\in\calE_\psi}\psi(|x^{ij}_t|) | v^i_t -v^j_t|^2 + \sigma^2\sum_{i=1}^N \lt|\sum_{j\in \calN_B^i }(v^j_t - v^i_t )\rt|^2 \leq -
\lambda\, \sum_{i,j=1}^N | v^i_t -v^j_t|^2\qquad\text{almost surely},
\]
with $\lambda=\psi^{min} L_{\calG_\psi} K - \sigma^2 c_B$, which by assumption is positive.

Now let $\tau_n:=\inf\{t:\|\bfv_t\|\ge n\}$, $n\in\bbN$, be a sequence of stopping times. Then equation \eqref{energy_eq} and the assumption $\lambda>0$ provide the following estimate for each $n\in\bbN$:
\[
	\bbE \lt[ \|\bfv_{t\wedge \tau_n}\|^2+\frac{M}{2} \sum_{(i,j)\in\calE_\phi}\Phi(|\bar x^{ij}_{t\wedge\tau_n}|^2) \rt] + 2N\lambda\, \bbE\lt(\int_0^{t\wedge \tau_n} \|\bfv_s\|^2\d s\rt)\le \bbE\lt[ \|\bfv_0\|^2+\frac{M}{2} \sum_{(i,j)\in\calE_\phi}\Phi(|\bar x^{ij}_0|^2) \rt]\,,
\]
from which we obtain
\[
	\bbP(\tau_n\le t) \le \frac{1}{n^2}\bbE \|\bfv_{t\wedge \tau_n}\|^2 \le \frac{1}{n^2}\bbE\lt[ \|\bfv_0\|^2+\frac{M}{2} \sum_{(i,j)\in\calE_\phi}\Phi(|\bar x^{ij}_0|^2) \rt]\longrightarrow 0\qquad\text{as $n\to\infty$}\,.
\]
Since $(\tau_n)_{n\in\bbN}$ is a monotonically increasing sequence, the convergence above shows that $\tau_n \to\infty$ almost surely as $n\to\infty$. Therefore, using Fatou's lemma, we may pass to the limit $n\to\infty$ to obtain \eqref{est_ev}.

As for the second statement, we first write \eqref{energy_eq} as
\[
	\d\|\bfv_t\|^2 = \bigl(f_t + g_t + h_t\bigr)\d t + \text{local martingale}\,.
\]
For the first two terms, we easily estimate 
\begin{align*}
\bbE\,|f_t|  
\leq 2NK\psi^{max} \bbE\|\bfv_t\|^2\,,\qquad
\bbE\,|g_t| \leq 2\sigma^2 c_B N\,\bbE\|\bfv_t\|^2\,.
\end{align*}
This together with \eqref{est_ev} yields the uniform-in-time bounds on $\bbE|f_t|$ and $\bbE|g_t|$. 

As for the estimate of $\bbE|h_t|$, we notice from \eqref{est_ev} that 
\[
	\bbE\,\Phi(|\bar x_t^{ij}|^2) \leq \frac 2M \bbE\|\bfv_0\|^2 + \sum_{(i,j)\in\calE_\phi} \bbE\,\Phi(|\bar x_0^{ij}|^2)\qquad\text{for all $(i,j)\in\calE_\phi$. }
\]
On the other hand, since $r\mapsto\Phi(r)$ is monotonically increasing, assumption \eqref{asp} provides the existence of some $\eta>0$ such that
\[
	\frac 2M \bbE\|\bfv_0\|^2 + \sum_{(i,j)\in\calE_\phi} \bbE\,\Phi(|\bar x_0^{ij}|^2) \le \Phi(\eta)\,.
\]
In particular, we have that
\[
	0\le \bbE\, \int_{|\bar x^{ij}_t|^2}^{\eta} \phi(r)\d r \le \phi^{max}\, \bbE \Bigl[\eta-|\bar x^{ij}_t|^2\Bigr]\qquad\text{for all $(i,j)\in\calE_\phi$\,,}
\]
which consequently yields the estimate
\[
	\max_{(i,j) \in \calE_\phi} \bbE\, |\bar x^{ij}_t|^2 \le \eta\,,
\]
and hence,
\begin{align}\label{bdd_x}
	\bbE \sum_{(i,j) \in \calE_\phi} |\bar x^{ij}_t|^2 \le |\calE_\phi| \eta\qquad\text{for all $t\ge 0$}\,.
\end{align}
We then estimate
$$\begin{aligned}
\bbE\,|h_t| &\leq M\phi^{max} \bbE \sum_{(i,j) \in \calE_\phi} |\bar x^{ij}_t| |v^j_t - v^i_t|  
\le  M\phi^{max}  \lt( \frac{|\calE_\phi| \eta}{2} + N\,\bbE \|\bfv_t\|^2\rt).
\end{aligned}$$
Equation \eqref{est_ev} again implies that $\bbE|h_t|$ is uniformly bounded in time. 

Using the previously defined stopping time $\tau_n$, we find for any $0<s<t<\infty$ and $n\in\bbN$:
\begin{align*}
	\bbE\|\bfv_{t\wedge \tau_n}\|^2 = \bbE\|\bfv_0\|^2 + \bbE \lt[\int_0^{t\wedge\tau_n} \bigl(f_r + g_r + h_r\bigr)\d r\rt].
\end{align*}
By means of the Lebesgue dominated convergence and the estimates derived above, we may pass to the limit $n\to\infty$ to obtain the desired identity \eqref{iden_ev}. From the estimates obtained for $\bbE|f_t|$, $\bbE|g_t|$, and $\bbE|h_t|$, we easily deduce that the map $(0,\infty)\ni t\mapsto\bbE \|\bfv_t\|^2$ is Lipschitz, therewith concluding the proof.
\end{proof}

\begin{remark}
	Notice that the estimate \eqref{est_ev} provided in Lemma~\ref{lem_energy} extends the well-posedness result of Proposition~\ref{prop:well-posedness} to all positive times $t\ge 0$.
\end{remark}

%%%%%%%%%%%%%%%%%%%%%%%%%%%%%%%%%%%%%%%%%%%%%%%%%%
%
%
%
% \section{Stochastic flocking estimate}
%
%
%%%%%%%%%%%%%%%%%%%%%%%%%%%%%%%%%%%%%%%%%%%%%%%%%%

\section{Stochastic flocking and pattern formation}\label{sec:flock}

\subsection{Stochastic flocking estimate}

In this subsection, we provide the details of proof for the stochastic flocking estimate in Theorem \ref{main_thm}.

\begin{proof}[Proof of Theorem \ref{main_thm}] 
	Letting $t\to\infty$ in the estimate \eqref{est_ev} in Lemma~\ref{lem_energy} shows that
	\[
		\int_0^\infty\bbE \|\bfv_t\|^2\d t < \infty\,.
	\]
Since the Lipschitz continuity of $(0,\infty)\ni t\mapsto \bbE\|\bfv_t\|^2$ implies its uniform continuity, we conclude 
\[
\bbE \|\bfv_t\|^2 \to 0\qquad\text{as $t \to \infty$\,.}
\]
As for the uniform bound on the relative distances, we use the bound \eqref{bdd_x} and Lemma~\ref{lem_conn} to deduce
$$\begin{aligned}
\sum_{i,j=1}^N \bbE |x^i_t - x^j_t|^2 %&\leq 2 \sum_{i,j=1}^N \bbE |\bar x^i_t - \bar x^j_t|^2 + 2 \sum_{i,j=1}^N  \bbE |z^i - z^j|^2\cr
&\leq \frac{2}{L_{\calG_\phi}} \sum_{(i,j) \in \calE_\phi} \bbE |\bar x^{ij}_t|^2 + 2 \sum_{i,j=1}^N |z^i - z^j|^2\cr
&\leq \frac{2|\calE_\phi|\eta}{L_{\calG_\phi}} + 2 \sum_{i,j=1}^N  |z^i - z^j|^2 < \infty \qquad \mbox{for all $t\ge 0$\,.}
\end{aligned}$$
This completes the proof.
\end{proof}

%%%%%%%%%%%%%%%%%%%%%%%%%%%%%%%%%%%%%%%%%%%%%%%%%%
%
%
%
% \section{Stochastic pattern formation}
%
%
%%%%%%%%%%%%%%%%%%%%%%%%%%%%%%%%%%%%%%%%%%%%%%%%%%
\subsection{Stochastic pattern formation}

In this part, we provide the details on the exponential emergence of the stochastic pattern formation in Theorem \ref{main_thm2}. For this, we introduce an energy functional 
\begin{align*}
	\calJ_{\alpha,\beta}(\bfx,\bfv) := \alpha \sum_{(i,j)\in\calE_\phi}\bbE\,\Phi(|\bar x^{ij}|^2) + \sum_{i=1}^N \bbE\lal  \bar x^i- \bar x^{ave}, v^i \ral + \beta \bbE\|\bfv\|^2,
\end{align*}
where $\alpha, \beta >0$ will be determined appropriately later. We will see in the following that for appropriate choices of $\alpha$ and $\beta$, the functional $\calJ_{\alpha,\beta}$ plays the role of a Lyapunov function for our stochastic system \eqref{main_eq}.

 We also recall the energy functional
\[
	\calH(\bfx,\bfv) = \bbE[\|\bar\bfx- \bar \bfx^{ave}\|^2 + \|\bfv\|^2]\,,
\]
defined in Section~\ref{sec:pre}. We remind the reader that, without loss of generality, $v^{ave}$ is assumed to be zero.

\medskip

We first show a good relationship between the functional $\calJ_{\alpha,\beta}$ and $\calH$.

\begin{lemma}\label{lem_J} Assume $\phi^{min} > 0$. For any $\alpha,\beta>0$ satisfying
\begin{equation}\label{alpha-beta}
	\alpha\beta > \frac{1}{4N \phi^{min}L_{\calG_\phi}}\,,
\end{equation}
there exist constants $c_0, c_1>0$ such that
\bq\label{eqv_HJ}
c_0\calH(\bfx,\bfv)\leq \calJ_{\alpha,\beta}(\bfx,\bfv) \leq c_1\calH(\bfx,\bfv)\,.
\eq
\end{lemma}
\begin{proof} Let us first recall $\Phi(r)=\int_0^r \phi(s)\d s$. Since $\phi^{min} > 0$, we easily find $\Phi(|\bar x^{ij}|^2) \ge \phi^{min}|\bar x^{ij}|^2$, and this gives
\[
\bbE\sum_{(i,j)\in\calE_\phi}\Phi(|\bar x^{ij}|^2) \ge \phi^{min}L_{\calG_\phi}\bbE\sum_{i,j=1}^N|\bar x^{ij}|^2.
\]
On the other hand, we get by symmetry
\[
\sum_{i,j=1}^N|\bar x^{ij}|^2 = \sum_{i,j=1}^N | \bar x^i - \bar x^{ave} - (\bar x^j - \bar x^{ave})|^2 = 2N \sum_{i=1}^N | \bar x^i - \bar x^{ave}|^2,
\]
and thus
\[
\bbE\sum_{(i,j)\in\calE_\phi}\Phi(|\bar x^{ij}|^2) \ge 2N \phi^{min}L_{\calG_\phi} \bbE\|\bar\bfx - \bar \bfx^{ave}\|^2.
\]
We also use Young's inequality to estimate
\[
\lt| \sum_{i=1}^N \bbE\lal  \bar x^i- \bar x^{ave}, v^i \ral \rt| \leq \frac{1}{2\beta} \bbE\|\bar\bfx - \bar \bfx^{ave}\|^2 + \frac\beta2  \bbE \|\bfv\|^2.
\]
Combining all of the estimates above, we have
\[
\calJ_{\alpha,\beta}(\bfx,\bfv) \ge \lt(2N \phi^{min}L_{\calG_\phi} \alpha - \frac{1}{2\beta}\rt)\bbE\|\bar\bfx- \bar \bfx^{ave}\|^2 + \frac\beta2\bbE\|\bfv\|^2.
\]
Hence, for any $\alpha,\beta >0$ satisfying \eqref{alpha-beta}, we obtain
\[
\calJ_{\alpha,\beta}(\bfx,\bfv) \geq \frac{1}{2\beta}\min\Bigl\{4\alpha\beta N \phi^{min}L_{\calG_\phi} - 1, \ \beta^2 \Bigr\}\, \bbE[\|\bar\bfx - \bar \bfx^{ave}\|^2 + \|\bfv\|^2].
\]
For the upper bound estimate on $\calJ_{\alpha,\beta}$, we get
\[
\bbE\sum_{(i,j)\in\calE_\phi}\Phi(|\bar x^{ij}|^2) \leq \phi^{max}\bbE\sum_{i,j=1}^N|\bar x^{ij}|^2 = 2N \phi^{max}\bbE\|\bar\bfx - \bar \bfx^{ave}\|^2,
\]
hence 
\begin{align*}
\calJ_{\alpha,\beta}(\bfx,\bfv) &\leq \lt(2N \phi^{max}\alpha + \frac{1}{2}\rt)\bbE\|\bar\bfx - \bar \bfx^{ave}\|^2 + \lt(\beta + \frac12 \rt) \bbE\|\bfv\|^2\cr
&\leq \max\lt\{2N \phi^{max}\alpha + \frac{1}{2}, \ \beta+\frac12 \rt\}\bbE[\|\bar\bfx - \bar \bfx^{ave}\|^2 + \|\bfv\|^2].
\end{align*}
This completes the proof.
\end{proof}

We are now in a position to prove Theorem \ref{main_thm2}. The proof is based on {\it hypocoercivity}-type estimates.

\begin{proof}[Proof of Theorem \ref{main_thm2}] It follows from the proof of Lemma \ref{lem_energy}  that 
\[
\frac{\d}{\d t}\sum_{(i,j)\in\calE_\phi} \bbE\,\Phi(|\bar x^{ij}_t|^2) = 2\bbE\sum_{(i,j) \in \calE_\phi} \phi(|\bar x^{ij}_t|^2) \lal \bar x^j_t - \bar x^i_t, v^j_t - v^i_t \ral 
\]
and
\begin{align*}
\frac{\d}{\d t} \bbE \|\bfv_t\|^2  &\leq -2N (\psi^{min} L_{\calG_\psi} K - \sigma^2 c_B ) \bbE \|\bfv_t\|^2  - M \bbE\sum_{(i,j) \in \calE_\phi} \phi(|\bar x^{ij}_t|^2) \lal \bar x^j_t - \bar x^i_t, v^j_t - v^i_t \ral.
\end{align*}
Applying It\^o's formula, we also readily find
\begin{align*}
\frac{\d}{\d t} \sum_{i=1}^N \bbE\lal  \bar x^i_t - \bar x_t^{ave}, v^i_t \ral &= \bbE \|\bfv_t\|^2  + \frac K2 \bbE\sum_{(i,j) \in \calE_\psi} \psi(|x^{ij}_t|) \lal v^j_t - v^i_t, \bar x^i_t - \bar x^j_t  \ral - \frac M2 \bbE\sum_{(i,j) \in \calE_\phi} \phi(|\bar x^{ij}_t|^2) | \bar x^j_t - \bar x^i_t |^2\cr
&\leq \bbE \|\bfv_t\|^2 + \frac{K\psi^{max}}{4\gamma} \bbE \sum_{(i,j) \in \calE_\psi}|v^j_t - v^i_t|^2 + \frac{K\psi^{max}\gamma}{4} \bbE \sum_{(i,j) \in \calE_\psi} |\bar x^j_t - \bar x^i_t|^2 \cr
&\quad - \frac{M\phi^{min}}{2} \bbE\sum_{(i,j) \in \calE_\phi} | \bar x^j_t - \bar x^i_t |^2\cr
&\leq \lt(1 + \frac{KN\psi^{max}}{2\gamma} \rt)  \bbE \|\bfv_t\|^2 - N\lt( M \phi^{min}L_{\calG_\phi} - \frac{K\psi^{max}\gamma}{2}\rt) \bbE \|\bar\bfx_t - \bar \bfx^{ave}_t\|^2,
\end{align*}
where $\gamma > 0$ will be specified later. Combining all of the above estimates yields
\begin{align*}
\frac{d}{dt} \calJ_{\alpha,\beta}(\bfx_t,\bfv_t) &\leq -\lt(2\beta N (\psi^{min} L_{\calG_\psi} K - \sigma^2 c_B ) -  \lt(1 + \frac{KN\psi^{max}}{2\gamma} \rt) \rt) \bbE \|\bfv_t\|^2 \cr
&\quad - N\lt( M \phi^{min}L_{\calG_\phi} - \frac{K\psi^{max}\gamma}{2}\rt) \bbE \|\bar\bfx_t - \bar \bfx^{ave}_t\|^2\cr
&\quad + (2\alpha - \beta M) \bbE\sum_{(i,j) \in \calE_\phi} \phi(|\bar x^{ij}_t|^2) \lal \bar x^j_t - \bar x^i_t, v^j_t - v^i_t \ral.
\end{align*}
We then choose $\gamma = L_{\calG_\phi}M\phi^{min}/K\psi^{max}$ and $\alpha = \beta M/2$ to obtain
\begin{align*}
\frac{d}{dt} \calJ_{\alpha,\beta}(\bfx_t,\bfv_t) &\leq -\lt(2\beta N (\psi^{min} L_{\calG_\psi} K - \sigma^2 c_B ) -  \lt(1 + \frac{N (K \psi^{max})^2}{2L_{\calG_\phi}M\phi^{min}} \rt) \rt) \bbE \|\bfv_t\|^2 \cr
&\quad - \frac{NM \phi^{min}L_{\calG_\phi}}{2} \bbE \|\bar\bfx_t - \bar \bfx^{ave}_t\|^2.
\end{align*}
Finally choosing $\beta>0$ large enough such that 
\[
\beta > \max\lt\{\frac{2L_{\calG_\phi}M\phi^{min}+N (K \psi^{max})^2}{4NL_{\calG_\phi}M\phi^{min}(\psi^{min} L_{\calG_\psi} K - \sigma^2 c_B ) } , \frac{1}{\sqrt{2N M\phi^{min}L_{\calG_\phi}}}\rt\}
\]
while taking into account Lemma \ref{lem_J}, we have
\[ 
\frac{d}{dt} \calJ_{\alpha,\beta}(\bfx_t,\bfv_t) \leq -c_2\bbE[\|\bar\bfx_t - \bar \bfx^{ave}_t\|^2 + \|\bfv_t\|^2] \le -(c_2/c_1)\calJ_{\alpha,\beta}(\bfx_t,\bfv_t)
\]
for some $c_2 > 0$, where $c_0$ is appeared in Lemma \ref{lem_J}. By Gronwall's inequality we then easily deduce
\[
	\calJ_{\alpha,\beta}(\bfx_t,\bfv_t) \le \calJ_{\alpha,\beta}(\bfx_0,\bfv_0)\, e^{-(c_2/c_1)t}\,.
\]
The previous inequality and Lemma \ref{lem_J} again yield
\[
\bbE[\|\bar\bfx_t - \bar \bfx^{ave}_t\|^2 + \|\bfv_t\|^2] \leq \frac{c_1}{c_0}\, \bbE[\|\bar\bfx_0 - \bar \bfx^{ave}_0\|^2 + \|\bfv_0\|^2]\, e^{- (c_2/c_1) t}\qquad\text{for all $t \geq 0$\,.}
\]
This concludes the proof. 
\end{proof}

%%%%%%%%%%%%%%%%%%%%%%%%%%%%%%%%%%%%%%%%%%%%%%%%%%
%
%
%
% \section{Numerical experiments}
%
%
%%%%%%%%%%%%%%%%%%%%%%%%%%%%%%%%%%%%%%%%%%%%%%%%%%
\section{Numerical experiments}\label{sec:num}

In this section, we present several numerical experiments regarding the dynamics of solutions to our main system \eqref{main_eq}, and compare the numerical results with the analytic stochastic flocking and pattern formation results found in Theorems~\ref{main_thm} and \ref{main_thm2}. We employed an improved Euler--Maruyama method proposed in \cite{R12} for the realization of the stochastic system \eqref{main_eq}.

Let us first introduce different network structures in Fig.\ \ref{fig1} that are mainly used in our numerical simulations. 
\begin{figure}[h]
	\centering
	\subfloat[$\calG_0$ for $N=30$]{
		\includegraphics[scale=1.3]{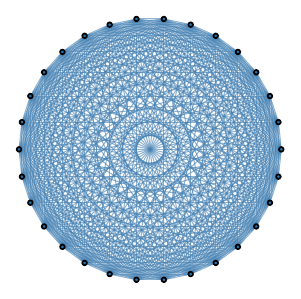}
		%\label{fig:net0}
		}
	\subfloat[$\calG_{1} $ for $N=30$]{
		\includegraphics[scale=1.3]{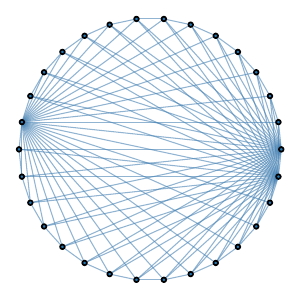}
		%\label{fig:net1}
		}
	\subfloat[$\calG_{2} $ for $N=30$]{
		\includegraphics[scale=1.3]{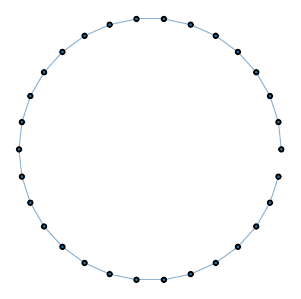}
		%\label{fig:net2}
		}
		
	\subfloat[$\calG_{3} $ for $N=30$]{
		\includegraphics[scale=1.3]{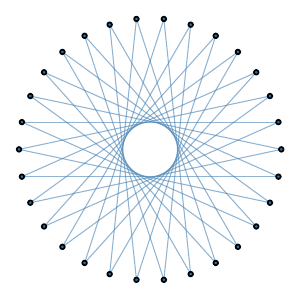}
		%\label{fig:net3}
		}
	\subfloat[$\calG_{3} $ for $N=40$]{
		\includegraphics[scale=1.3]{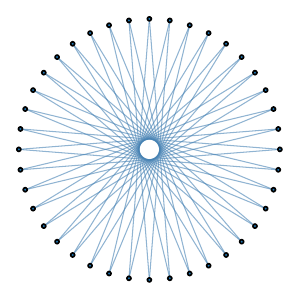}
		%\label{fig:net3-1}
		}
	\subfloat[$\calG_{3} $ for $N=150$]{
		\includegraphics[scale=0.4]{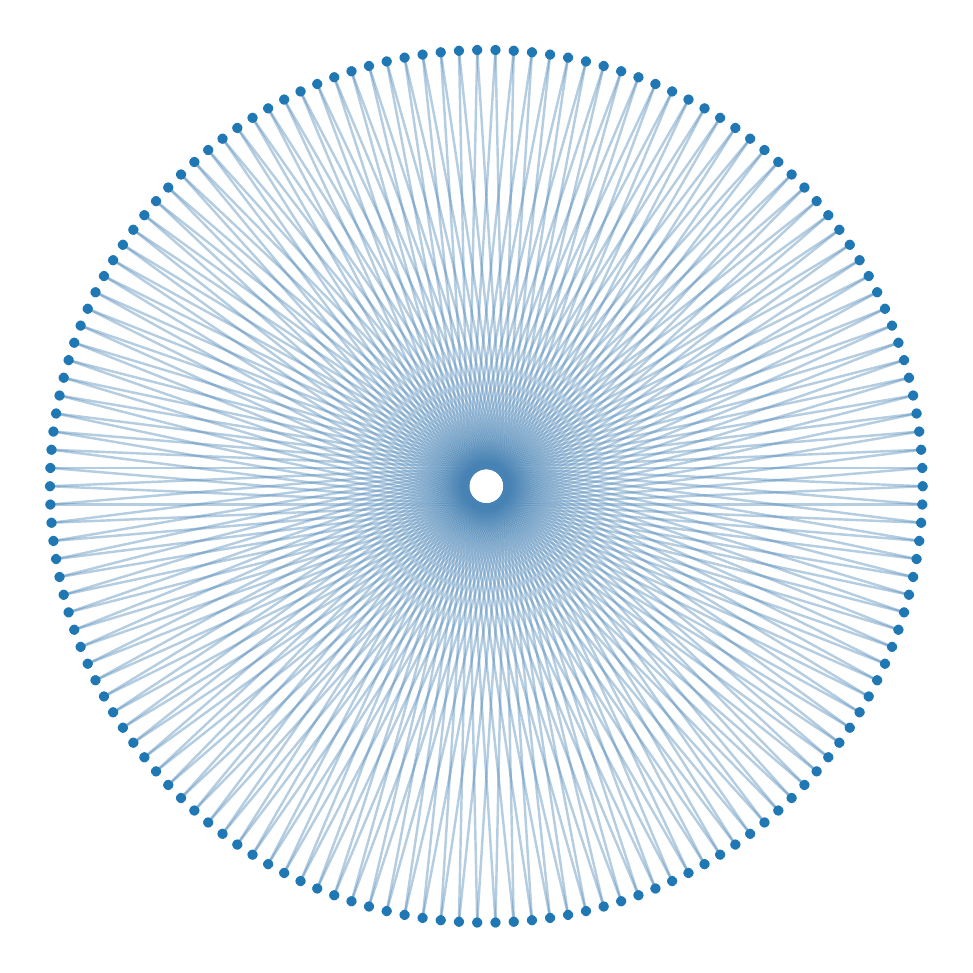}
		%\label{fig:net3-2}
		}
		
	\subfloat[$\calG_{4} $ for $N=30$]{
		\includegraphics[scale=1.3]{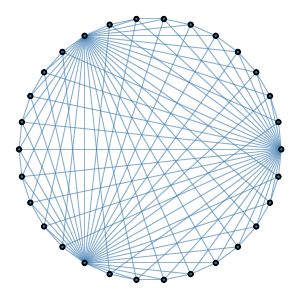}
		%\label{fig:net4}
		}
	\subfloat[$\calG_4$ for $N=50$]{
		\includegraphics[scale=1.3]{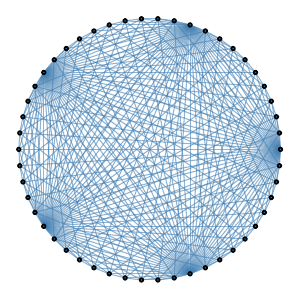}
		%\label{fig:net4-1}
		}
	\subfloat[$\calG_4$ for $N=150$]{
		\includegraphics[scale=1.3]{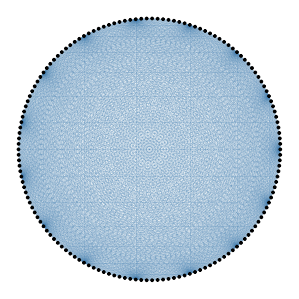}
		%\label{fig:net4-2}
		}
	\caption{Different network structures mainly used in simulations}
	\label{fig1}
\end{figure}
The graph $\calG_0$ in Fig.\ \ref{fig1} (a) is complete, i.e.\ every vertex is connected with every other vertex. In the graph $\calG_1$, there are only three vertices that are connected to all the other vertices and the rest are connected with its neighbors. Note that $\calG_0 = \calG_1$ if $N \leq 5$. $\calG_2$ is the graph where each vertex except the starting and terminal vertices is connected with its neighbors. This network topology is used in \cite{CKPP19} for the control signal $\bfu_t$. The edge set of $\calG_3$ consists of $(i,j)$ with $|i-j| \equiv n$ (mod $N$), where the positive integer $n$ is chosen as $2$, $N/2-1$, and $N/2-2$ when $N\not\equiv 0$ (mod 2), $N\equiv 0$ (mod 4), and $N\equiv 2$ (mod 4), respectively. Finally, $\calG_4$ is the graph where every vertex is connected with its neighbors and the vertex $i$ with $i \equiv 1$ (mod 10) is connected with every other vertex. We notice that all the graphs presented above are connected. Information on the graphs described above are summarized in Table \ref{table:net} below.

\begin{table}[h]
\begin{tabular}{|c|c|l|l|c|}
\hline
\multicolumn{1}{|C|}{\mathbfcal{G}} & \textbf{Fig.} & \multicolumn{1}{C|}{\mathbfcal{E}}                         & \multicolumn{1}{C|}{\boldsymbol{|\calE|}}                         & \multicolumn{1}{C|}{\boldsymbol{d(\calG)}} \\ \hline
	$\calG_0$ & (a) & $\{(i,j) :  i\neq j  \}$ & $N(N-1)$ & 1 \\  \hline 
	$\calG_1$ & (b) &  \begin{tabular}[c]{@{}L@{}} \{(i,j) : |i-j| = 1\} \\ \mbox{} \quad \cup \{(i,j) : i \mbox{ or } j=1,\lfloor N/2\rfloor, N\\ \mbox{} \hspace{1.1cm} \mbox{ with } i \neq j\} \end{tabular} &  \begin{tabular}[c]{@{}L@{}} 8N-22 \mbox{ when } N\geq 6, \\ N(N-1) \mbox{ when } N\leq 5 \end{tabular} & 2 \\ \hline
	$\calG_{2} $ & (c)  & $\{ (i,j) : |i-j|=1\}$ & $2(N-1)$ & $N-1$ \\ \hline
	$\calG_{3}$ &  (d),\,(e),\,(f) &  \begin{tabular}[c]{@{}L@{}} \{(i,j) : |i-j|\equiv n\mbox{ (mod N)} \} \mbox{ where} \\ n = \left\{ \begin{array}{ll}
 	2 & \text{if $N\not\equiv 0$ (mod 2)}\\
	 N/2-1 & \text{if $N\equiv 0$ (mod 4)}\\
	 N/2-2 & \text{if $N\equiv 2$ (mod 4)}
	  \end{array} \right.\end{tabular}  & $2N$ &  $\lfloor N/ 2 \rfloor$ \\ \hline
	$\calG_{4}$ & (g),\,(h),\,(i) & \begin{tabular}[c]{@{}L@{}} \{(i,j) : |i-j| = 1 \mbox{ or } N -1\}\\ \mbox{} \quad \cup \{(i,j) : i \mbox{ or } j \equiv 1 \mbox{ (mod 10)} \\ \mbox{} \hspace{1.1cm} \mbox{ with } i \neq j \} \end{tabular}
		&  \begin{tabular}[c]{@{}L@{}} 2N \lfloor (N+9) /10 \rfloor- \lfloor (N+9) /10 \rfloor ^2\\ \mbox{}\quad-5 \lfloor (N+9) /10 \rfloor+2N+2 \\ \mbox{} \hspace{1.1cm} \mbox{ when } N\equiv 1 \mbox{ (mod 10)}, \\ 2N \lfloor (N+9) /10 \rfloor- \lfloor (N+9) /10 \rfloor ^2\\ \mbox{}\quad -5 \lfloor (N+9) /10 \rfloor+2N \\  \mbox{} \hspace{1.1cm} \mbox{ when } N\not\equiv 1 \mbox{ (mod 10)}\end{tabular}  & 2 \\ \hline
\end{tabular}
\caption{A detailed information on the network structures}
\label{table:net}
\end{table}

\medskip

In the following two subsections, we consider two patterns, namely the Greek alphabet $\pi$ and Einstein's face image. These images are obtained from the search engine Wolfram Alpha\footnote{https://www.wolframalpha.com/}.

%%%%%%%%%%%%%%%%%%%%%%%%%%%%%%%%%%
%
%
% \subsection{$\pi $-like pattern} \label{pi-pattern}
%
%
%%%%%%%%%%%%%%%%%%%%%%%%%%%%%%%%%%

\subsection{Greek alphabet $\pi$ pattern} \label{pi-pattern}

Here, we consider a vector $\bfz$ in the control signal $\bfu_t$ that achieves the Greek alphabet $\pi$ shown in Fig.~\ref{fig_pi}(a). Throughout this subsection, we perform simulations for the system \eqref{main_eq} with $N=30$, $K=5$, $M=7$, and $\sigma=10^{-3}$. The communication weight functions are chosen as 
\begin{align}\label{num:psi-phi}
	\psi(r)=(1+r^2)^{-0.25}+0.3\,,\qquad \phi(r)=(1+r^2)^{-0.25}+0.1\,.
\end{align}
For the network structures, we consider $\calG_{\psi }=\calG_{3}$, $\calG_{\phi }=\calG_1$, and $\calG_{B}=\calG_0$, see Fig.~\ref{fig_pi}. 

\begin{figure}[h]
	\centering 
	\subfloat[graph of $\pi$ curve]{
		\includegraphics[scale=0.6]{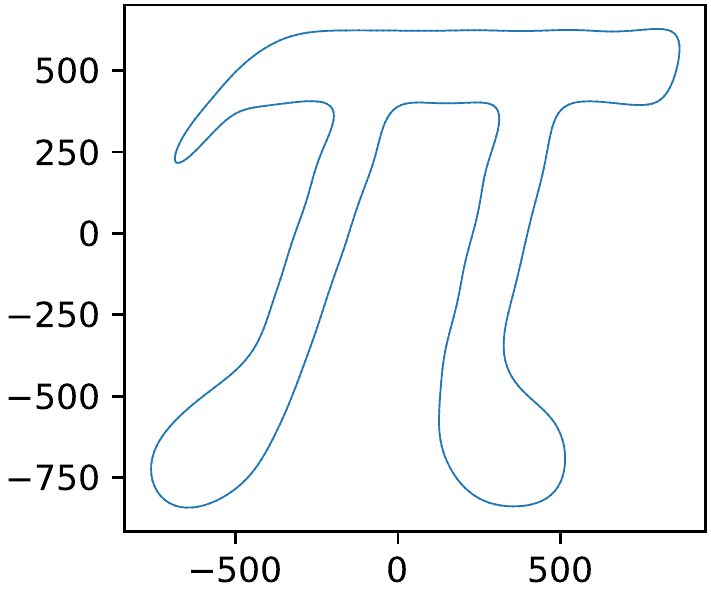}
		\label{fig:pi-curve}
	}
	\subfloat[$\calG_0$ in $\pi$ curve]{
		\includegraphics[scale=0.6]{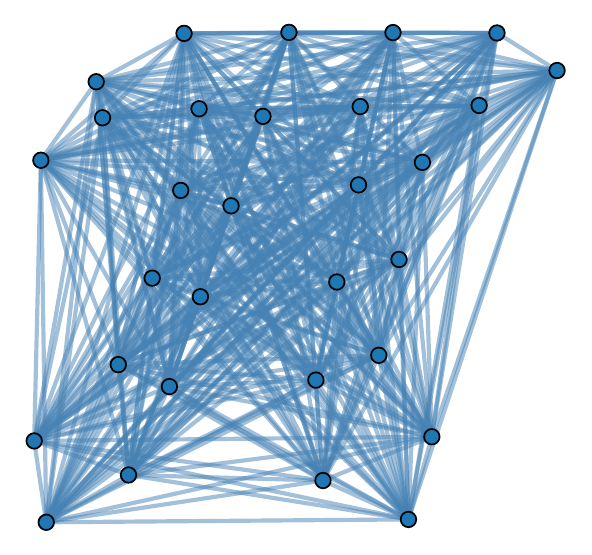}
		\label{fig:pi-net0}
	}
	\subfloat[$\calG_{1} $ in $\pi$ curve]{
		\includegraphics[scale=0.6]{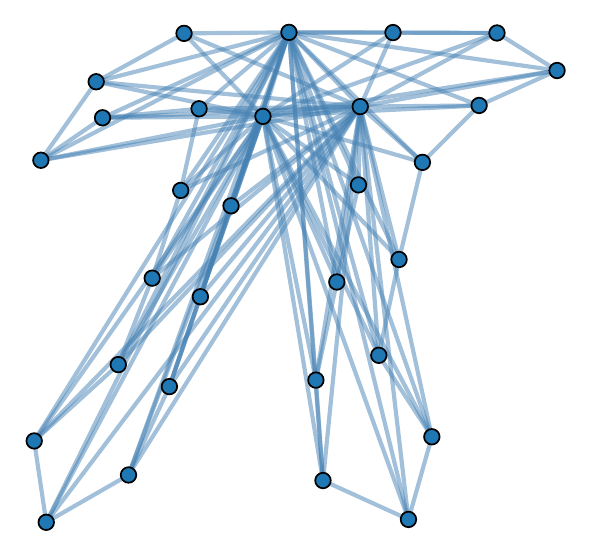}
		\label{fig:pi-net1}

	}
	\subfloat[$\calG_{3} $ in $\pi$ curve]{
		\includegraphics[scale=0.6]{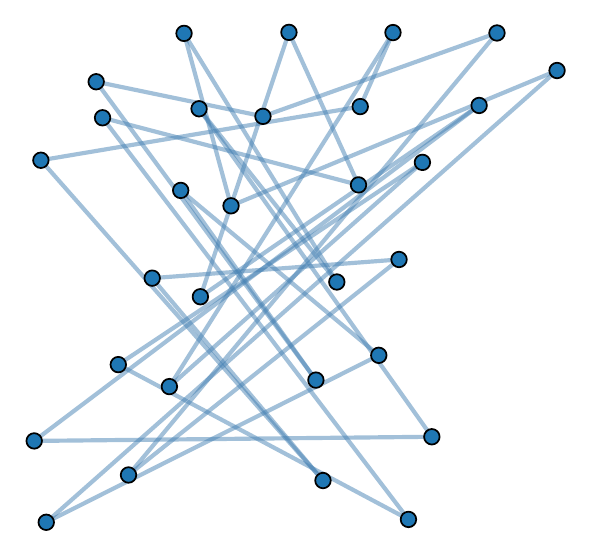}
		\label{fig:pi-net4}
	}

	\caption {
		Greek alphabet $\pi$ pattern and networks
	}
	\label{fig_pi}
\end{figure}

The initial positions of particles are randomly distributed in the box $[-212.125,212.125]^2$, and then modified to satisfy $\bbE\,x_{0}^{ave}= z^{ave}$. The initial velocity of particles are randomly distributed in the box $[-25,25]^2$ and similarly as above, it is modified to make the average $\bbE\,v_0^{ave}=0$. 

In our simulations, we find
\[
K\psi^{min}\geq 1.515\geq 2\sigma^2\frac{N}{L_{\calG_{3}}} \geq 2\sigma^2\frac{N}{L_{\calG_1}}\geq 2\sigma^2 \frac{N}{L_{\calG_0}}, \qquad \mbox{and} \qquad \int_{0}^{\infty}\phi(r)\,dr=\infty.
\] 
Thus our initial configurations for the simulations satisfy the assumptions in Theorems~\ref{main_thm}  and \ref{main_thm2}. We compute 100 realizations of the solution with the same initial data.  

 Fig.~\ref{fig_pi3} illustrates the $\pi$ pattern formation of solutions in the plane at different times. The simulations show that particles first form the desired pattern, and the their velocities are aligned later; small fluctuations in velocity are observed.

\begin{figure}[h]
	\centering
	\subfloat[t=0.000]{
		\includegraphics[scale=0.7]{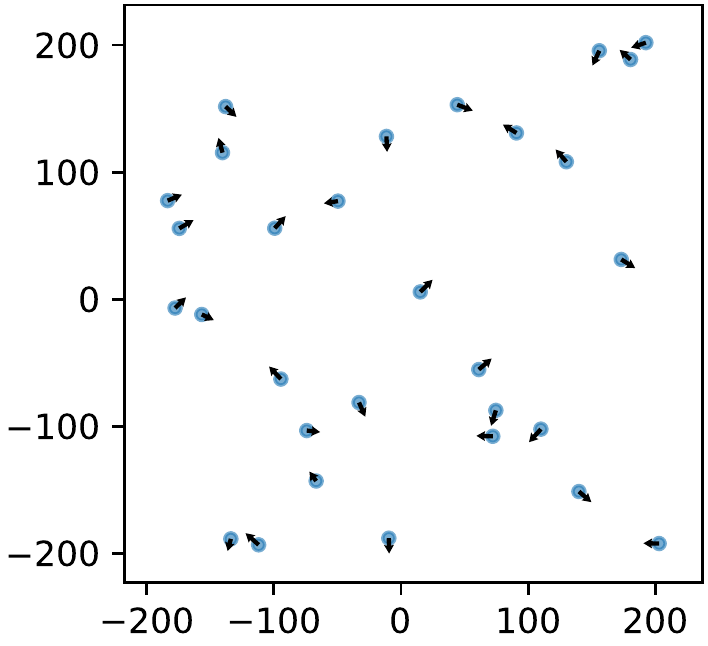}
		\label{fig:pi1}                    
	}                                          
	\subfloat[t=2.150]{                               
		\includegraphics[scale=0.7]{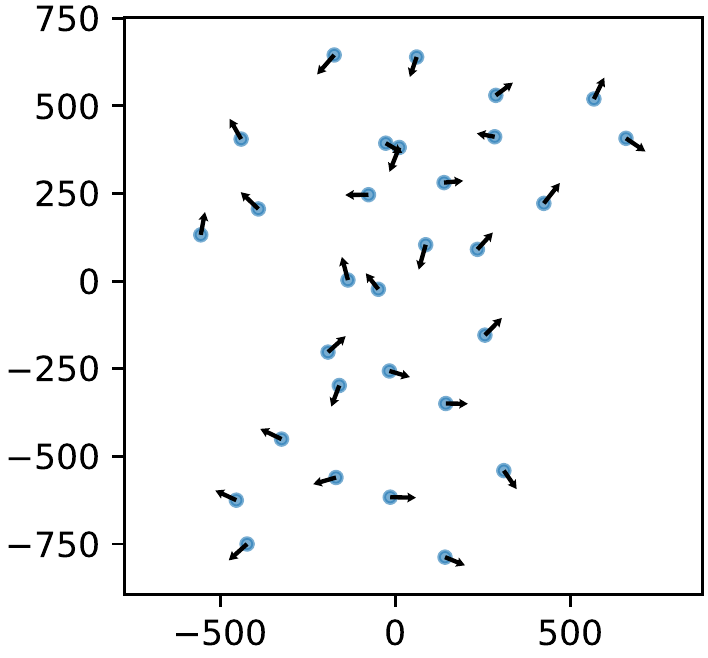}
		\label{fig:pi2}                    
	}                                                                               
	\subfloat[t=5.175]{                               
		\includegraphics[scale=0.7]{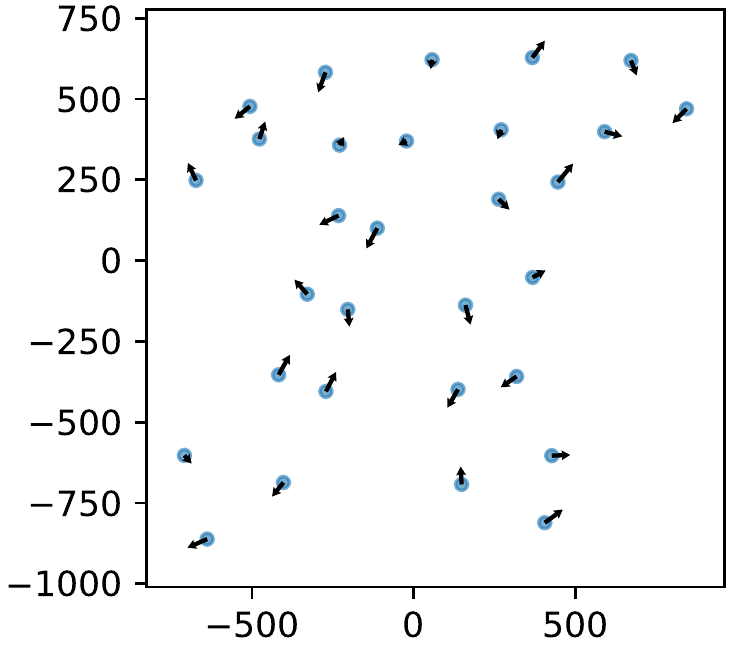}
		\label{fig:pi3}                    
	}           
	                               
	\subfloat[t=15.525]{                               
		\includegraphics[scale=0.7]{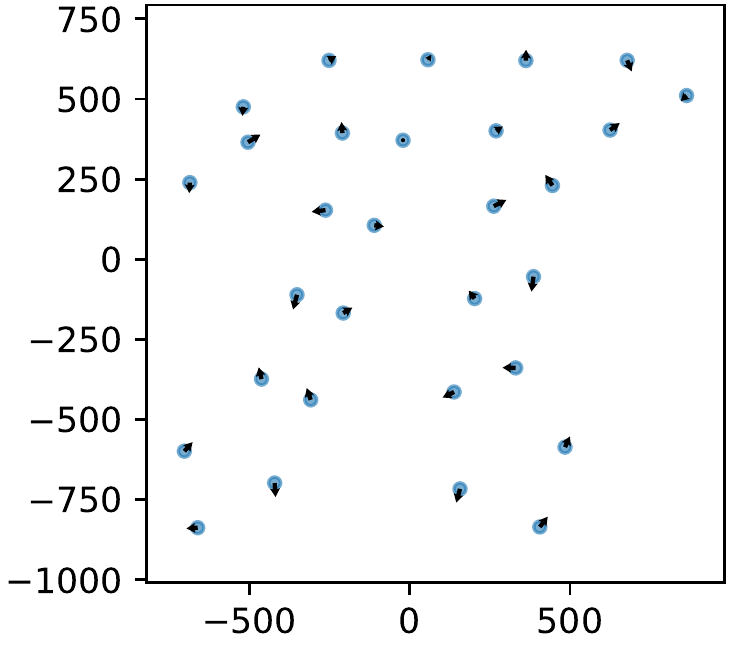}
		\label{fig:pi4}                    
	}                                                                                     
	\subfloat[t=24.450]{                               
		\includegraphics[scale=0.7]{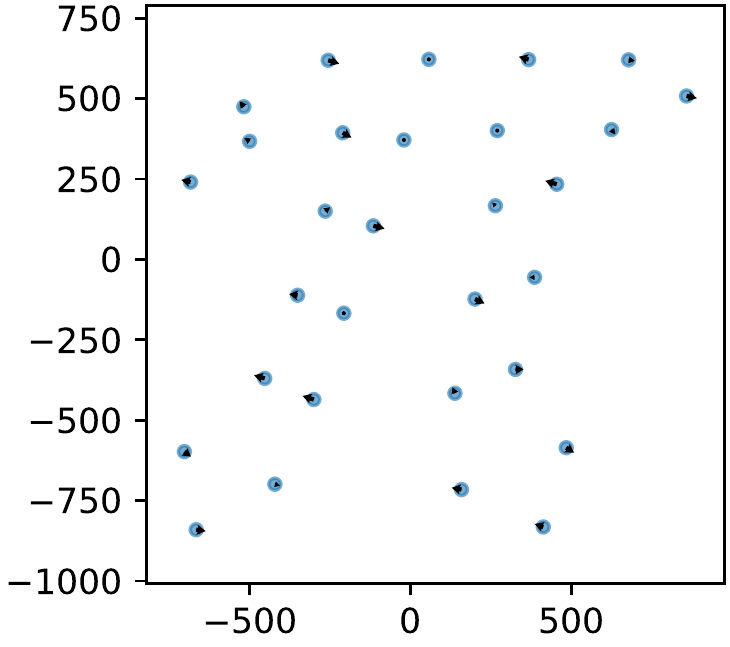}
		\label{fig:pi5}                    
	}                                          
	\subfloat[t=34.975]{                               
		\includegraphics[scale=0.7]{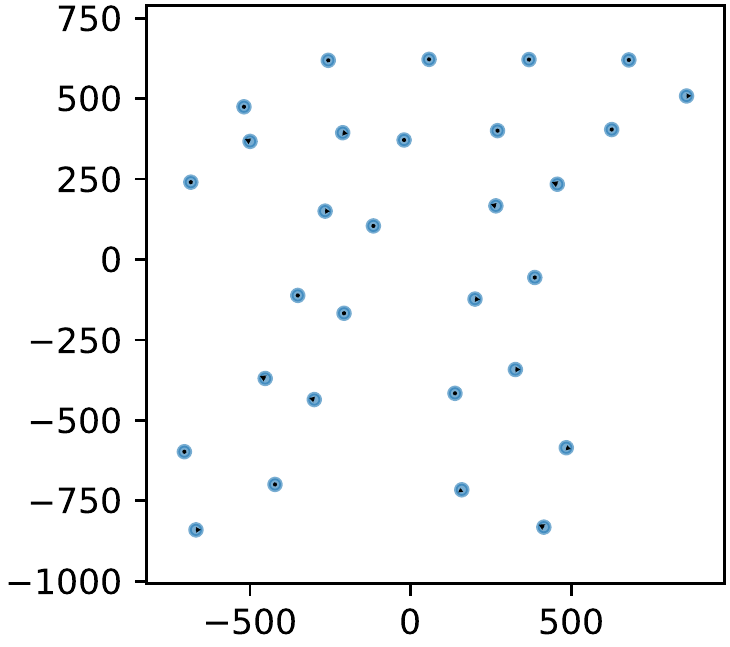}
		\label{fig:pi6}
	}

	\caption{
		Snapshots for $\pi$ pattern simulation}
		\label{fig_pi3}
\end{figure}

\begin{table}[ht]
	\begin{tabular}{|c|c|c|c|c|}
\hline
$\mathbfcal{G_\phi}/\mathbfcal{G_\phi}$ & $\calG_0/\calG_0$ & $\calG_0/\calG_1$ & $\calG_2/\calG_0$ & $\calG_2/\calG_1$ \\
\hline
$\mathbfcal{\beta}$ ($\simeq$) & $9.50655 \cdot 10^{3}$ & $4.18548 \cdot 10^5$ & $7.72684 \cdot 10^6$ & $3.40218 \cdot 10^8$ \\
\hline
\end{tabular}
\caption{$\beta$ values for pairs of networks considered in Fig.\ \ref{fig_pi_comp1} and \ref{fig_pi_comp2} below.}\label{table:beta}
\end{table}

In Fig.~\ref{fig_pi_comp1}, we show the time evolutions of energy functionals $\calH$ and $\calJ_{\alpha,\beta}$ for different networks. To be precise, we fix the network structure for $\calG_\psi = \calG_B = \calG_0$ and vary the network structures for $\calG_\phi  = \calG_i, i=0,\ldots,4$. Fig.~\ref{fig_pi_comp1} (A) and (B) illustrate the comparison of the time evolutions of $\calH$ and $\calJ_{\alpha,\beta}$, respectively, with different networks $\calG_\phi$. In this scenario, we find that both energies decay monotonically. We also observe that control networks $\calG_\phi$ with diameter $d(\calG_\phi)=O(1)$ provide faster decay than ones with $d(\calG_\phi)=O(N)$. Fig.~\ref{fig_pi_comp1} (C) and (D) show the equivalence relation between the energy functionals $\calJ_{\alpha,\beta}$ and $\calH$, which is as expected from the analytical result obtained in Lemma \ref{lem_J} (cf.\  \eqref{eqv_HJ}). The $\beta$-value used for Fig.~\ref{fig_pi_comp1} (C) and (D), i.e.\ for the cases $\calG_\phi=\calG_0$ and $\calG_\phi=\calG_1$, are given in Table~\ref{table:beta}, from which the rest of the parameters $\alpha$, $c_i$, $i=0,1,2$, $p$ and $q$ are determined based on the value of $\beta$.

In Fig.~\ref{fig_pi_comp2}, we consider a sparse network for the interacting particle system, i.e.\ we take $\calG_\psi=\calG_B=\calG_2$. Fig.~\ref{fig_pi_comp2} (A) and (B) show the comparison of the time evolutions of $\calH$ and $\calJ_{\alpha,\beta}$, respectively, with different networks $\calG_\phi=\calG_i$, $i=0,\ldots,4$. Compared to the fully connected case $\calG_\psi = \calG_B = \calG_0$, we observe the formation of oscillations in the behavior of the energy functional $\calH$, which is not present for the energy functional $\calJ_{\alpha,\beta}$. Indeed, as expected from the proof of Theorem \ref{main_thm2}, Fig.~\ref{fig_pi_comp2} (B) illustrates the monotonic decrease of the energy functional $\calJ_{\alpha,\beta}$. Contrary to the fully connected case for $\calG_\psi$, it is unclear from the simulations that control networks $\calG_\phi$ with smaller diameter $d(\calG_\phi)$ provide a faster decay. Surprisingly, the control network structure $\calG_\phi=\calG_3$ proves to have a better decay behavior compared to the fully connected case $\calG_\phi=\calG_0$. This shows that the asymptotic behavior of solutions is strongly influenced by the initial configurations in case of a sparse network for the interacting particles. Nevertheless, Fig.~\ref{fig_pi_comp2} (C) and (D) show that the inequality \eqref{eqv_HJ} obtained in Lemma \ref{lem_J} is upheld. The $\beta$-value for these cases are found in Table~\ref{table:beta}.

\begin{figure}[h]
	\centering
	\subfloat[]	{
		\includegraphics[scale=0.6]{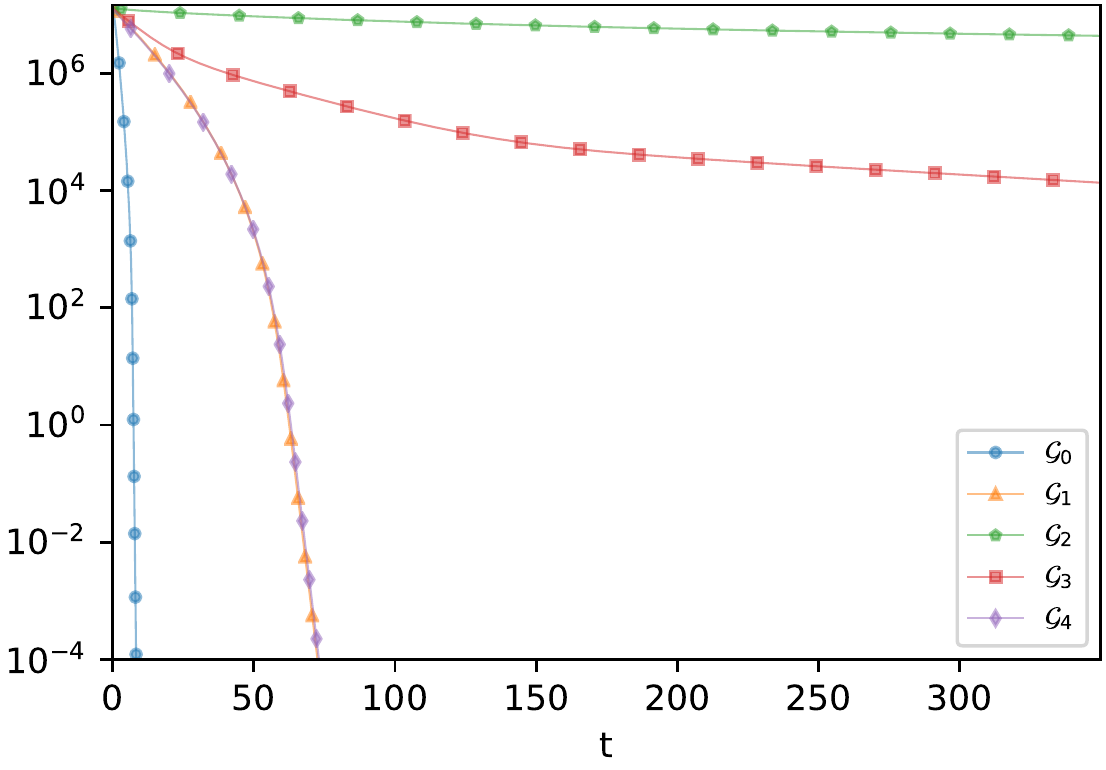}
		\label{fig:pi-comp-ps0-1}
	}
	\subfloat[]	{
		\includegraphics[scale=0.6]{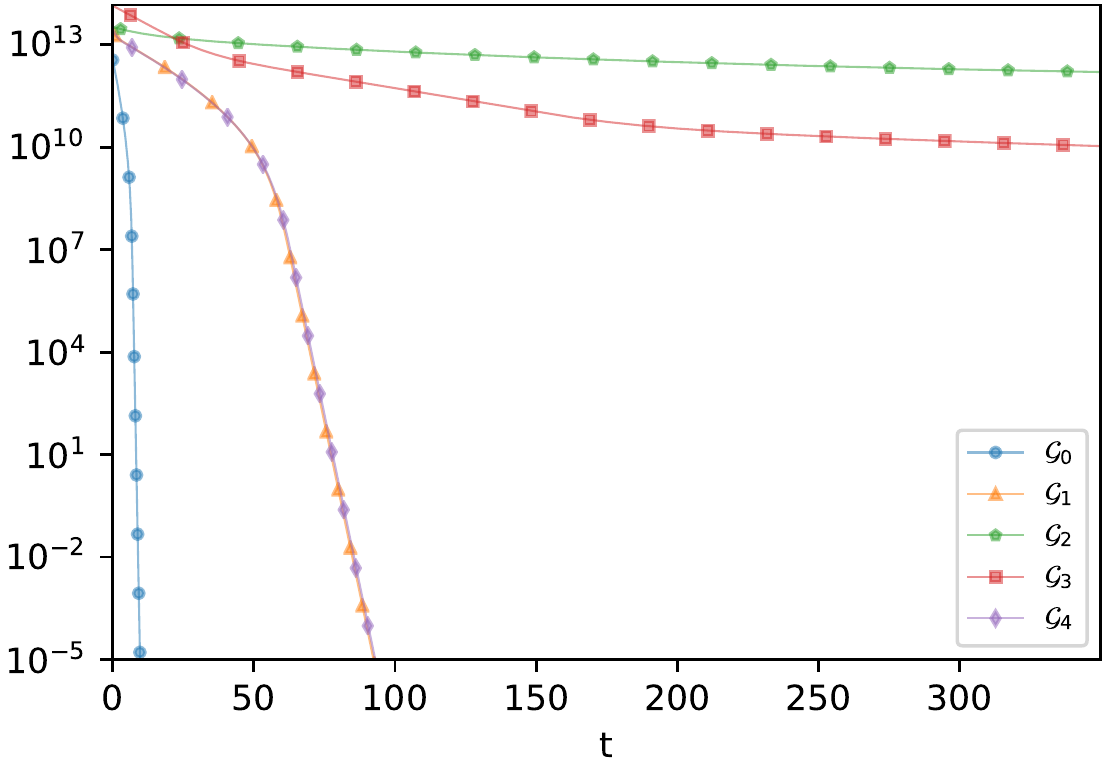}
		\label{fig:pi-comp-ps0-2}
	}

	\subfloat[]	{
		\includegraphics[scale=0.6]{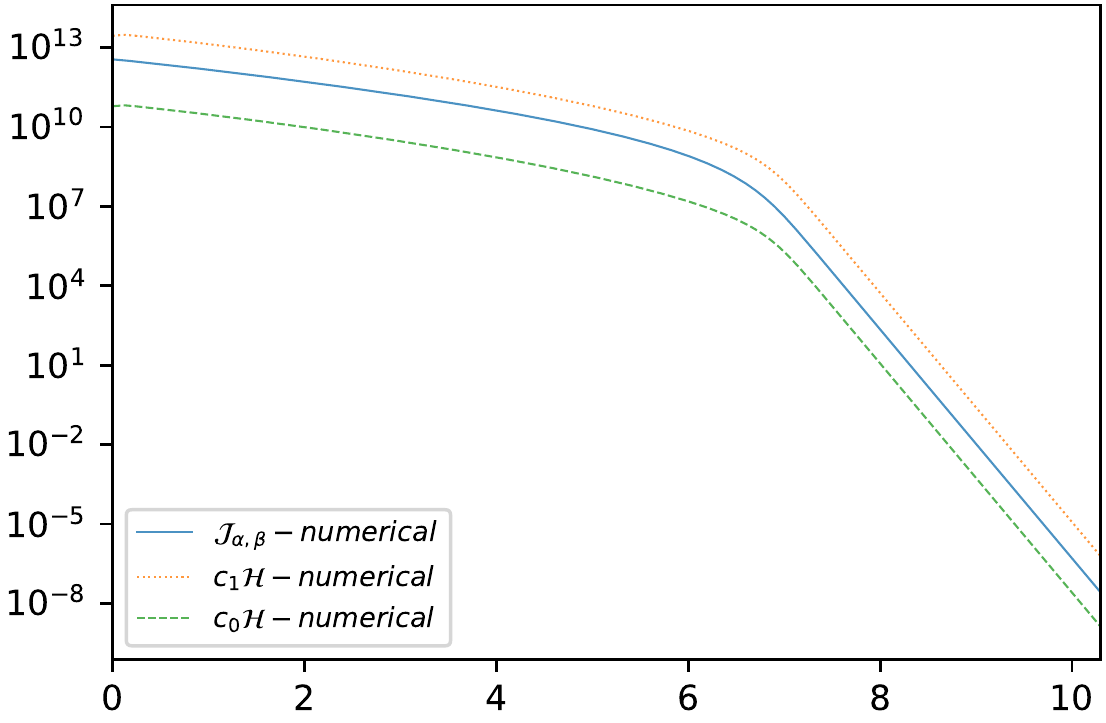}
		\label{fig:pi-comp-ps0-4}
	}
	\subfloat[]	{
		\includegraphics[scale=0.6]{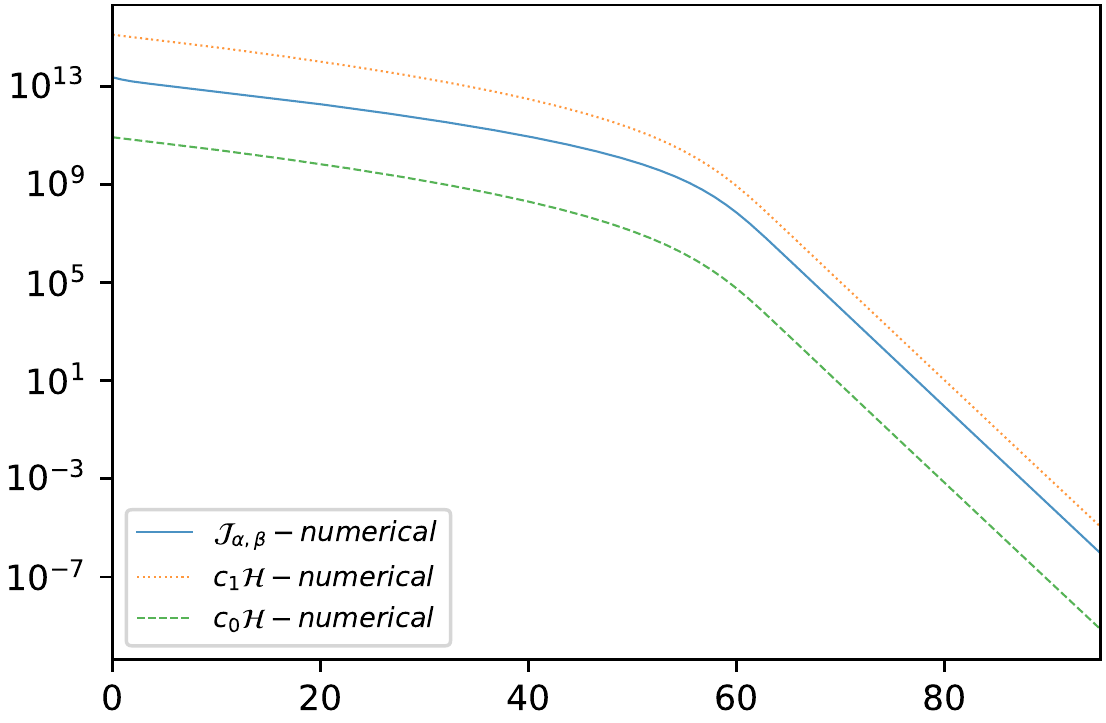}
		\label{fig:pi-comp-ps0-6}
	}
	\caption {Time evolutions of energy functionals $\calH$ and $\calJ_{\alpha,\beta}$ with $\calG_{\psi }=\calG_{B}=\calG_0  $ and $\calG_\phi  = \calG_i, i=0,1,\dots,4$. (A) Time evolution of $\calH$. (B) Time evolution of $\calJ_{\alpha,\beta}$. 
	(C) Comparison among $c_0 \calH$, $\calJ_{\alpha,\beta} $, and  $c_1 \calH$ when $\calG_{\phi }=\calG_{0} $. 
	(D) Comparison among $c_0 \calH$, $\calJ_{\alpha,\beta} $, and  $c_1 \calH$ when $\calG_{\phi }=\calG_1$. }  
	\label{fig_pi_comp1}
\end{figure}

\
\begin{figure}[h]
	\centering
	\subfloat[]	{
		\includegraphics[scale=0.6]{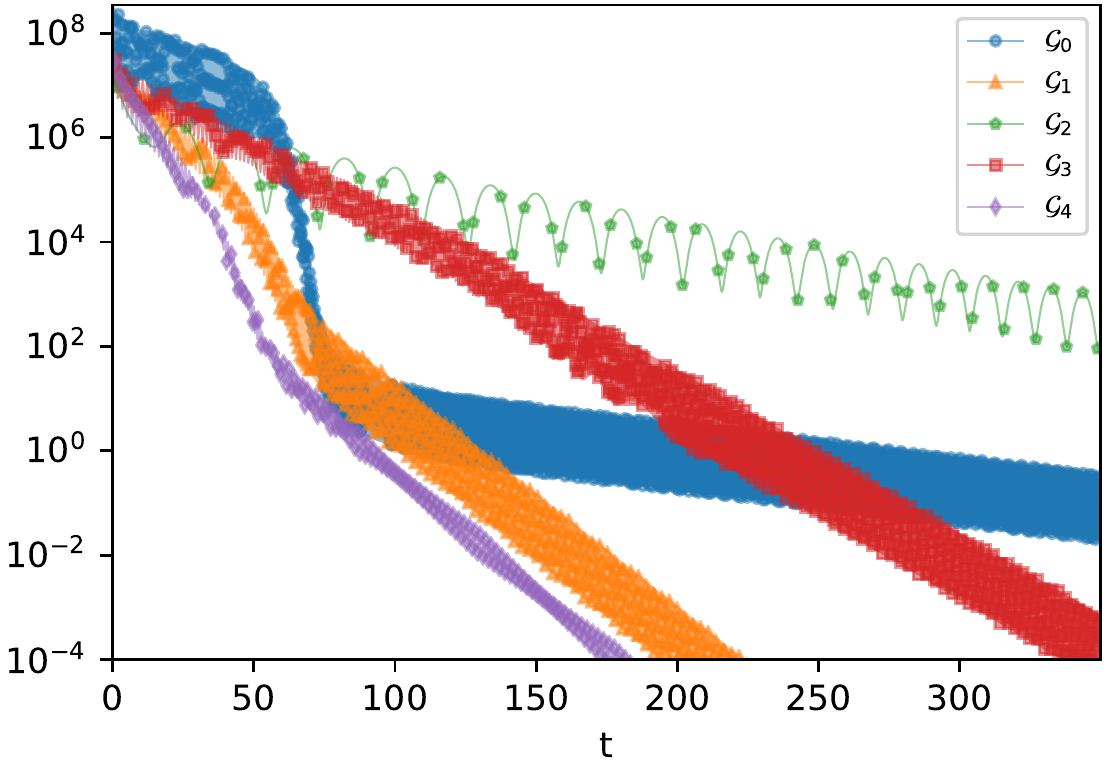}
		\label{fig:pi-comp-ps2-1}
	}
	\subfloat[]	{
		\includegraphics[scale=0.6]{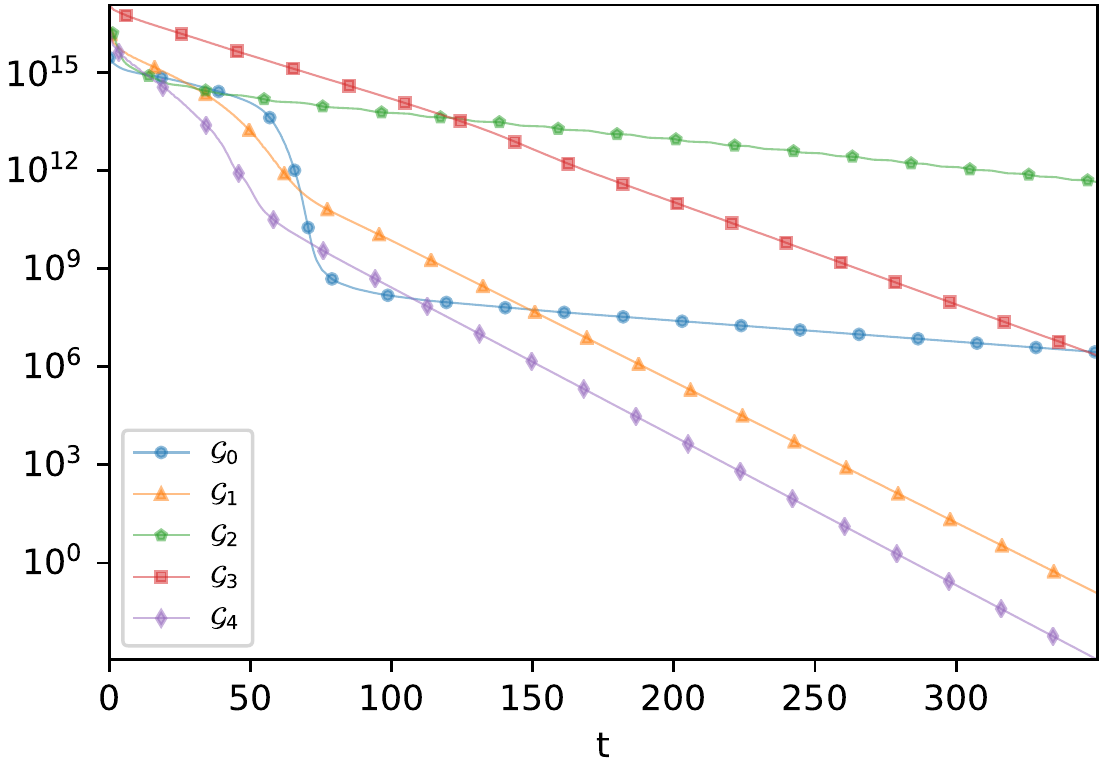}
		\label{fig:pi-comp-ps2-2}
	}

	\subfloat[]	{
		\includegraphics[scale=0.6]{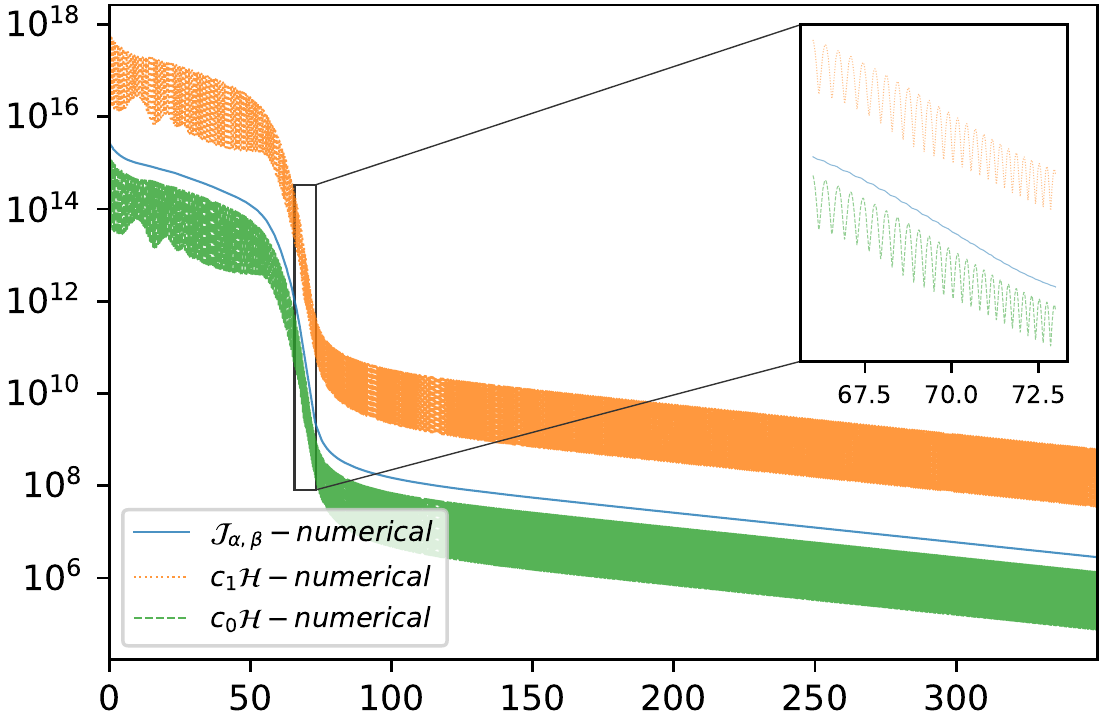}
		\label{fig:pi-comp-ps2-4}
	} 
	\subfloat[]	{
		\includegraphics[scale=0.6]{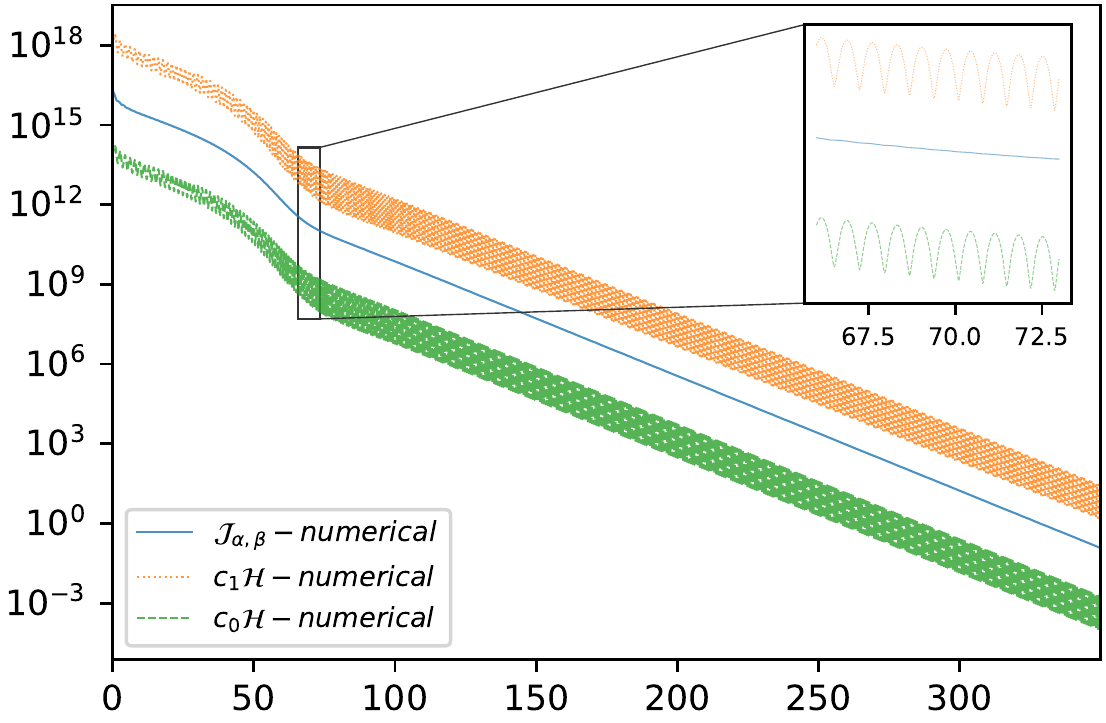}
		\label{fig:pi-comp-ps2-6}
	}
	\caption {Time evolutions of energy functionals $\calH$ and  $\calJ_{\alpha,\beta}$ with $\calG_{\psi }=\calG_{B}=\calG_2$ and $\calG_\phi  = \calG_i, i=0,1,\dots,4$. (A) Time evolution of $\calH$. (B) Time evolution of $\calJ_{\alpha,\beta}$. 
	(C) Comparison among $c_0 \calH$, $\calJ_{\alpha,\beta} $, and  $c_1 \calH$ when $\calG_{\phi }=\calG_{0}  $. 
	(D) Comparison among $c_0 \calH$, $\calJ_{\alpha,\beta} $, and  $c_1 \calH$ when $\calG_{\phi }=\calG_1$. 
} 
	\label{fig_pi_comp2}
\end{figure}

%%%%%%%%%%%%%%%%%%%%%%%%%%%%%%%%%%
%
%
% \subsection{$\pi $-like pattern} \label{pi-pattern}
%
%
%%%%%%%%%%%%%%%%%%%%%%%%%%%%%%%%%%

\subsection{Einstein's face image pattern}

In this part, we illustrate a pattern formation in the plane for Einstein's face image with $N=500$ particles. The parameters are chosen as $K=0.5$, $M=7$, and $\sigma=10^{-5}$, and the weight functions are selected as in \eqref{num:psi-phi} above. The interaction network structures are given as $\calG_{\phi}=\calG_{4}$, $\calG_{\psi }=\calG_4$, and $\calG_{B}=\calG_0$, and the vector $\bfz$ is chosen to have the desired image pattern, see Fig.~\ref{fig_ein0}. In this case, we easily verify the assumptions in Theorem \ref{main_thm}. 

The initial positions $\bfx_0$ and velocities $\bfv_0$ of particles are randomly distributed in $[-238.25,238.25]^2$ and $[-25,25]^2$, respectively, and, similarly as before, are modified to satisfy the conditions $\mathbb{E} x_0^{ave} = z^{ave}$ and $\mathbb{E}v_0^{ave} = 0$. This implies that the conditions in Theorem \ref{main_thm2} for the pattern formation are verified. 
\begin{figure}[h]
	\centering 
	\subfloat[Image of Einstein's face]{
		\includegraphics[scale=0.5]{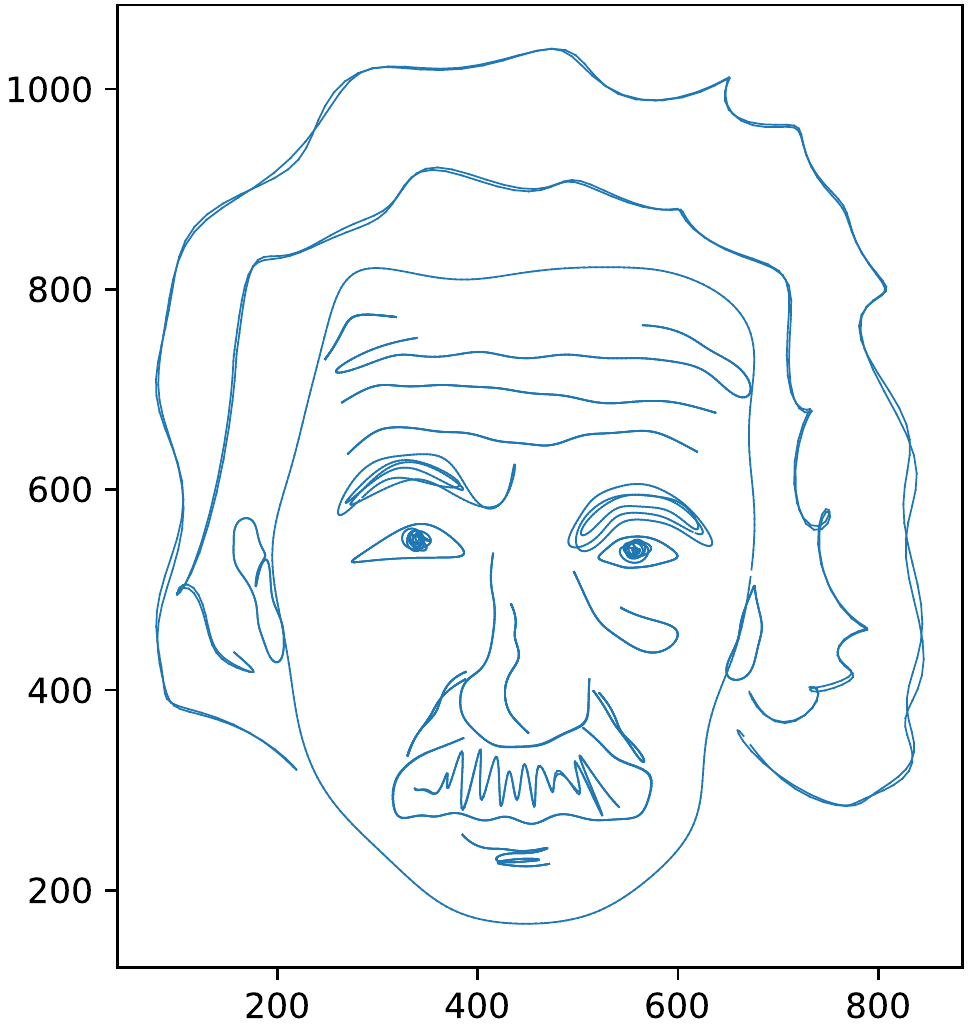}
		\label{fig:pi-curve}
	}
	\subfloat[$\calG_0$ in Einstein's face]{
		\includegraphics[scale=0.55]{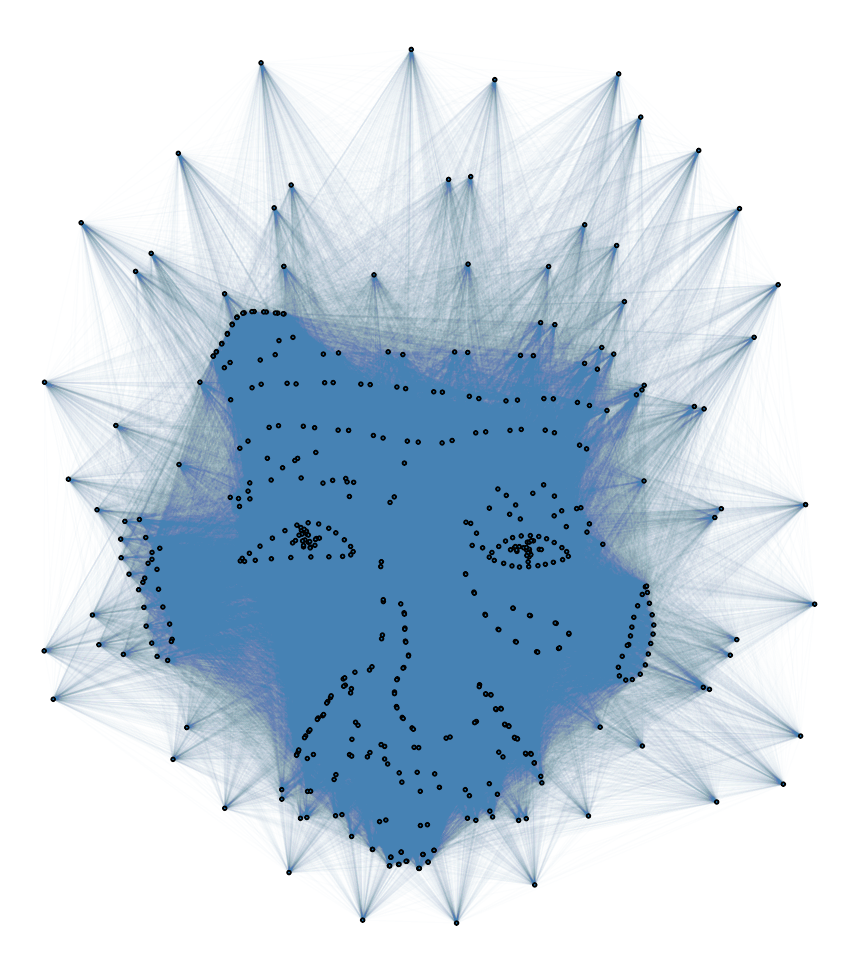}
		\label{fig:pi-net0}
	}
	\subfloat[$\calG_{4}$ in Einstein's face]{
		\includegraphics[scale=0.55]{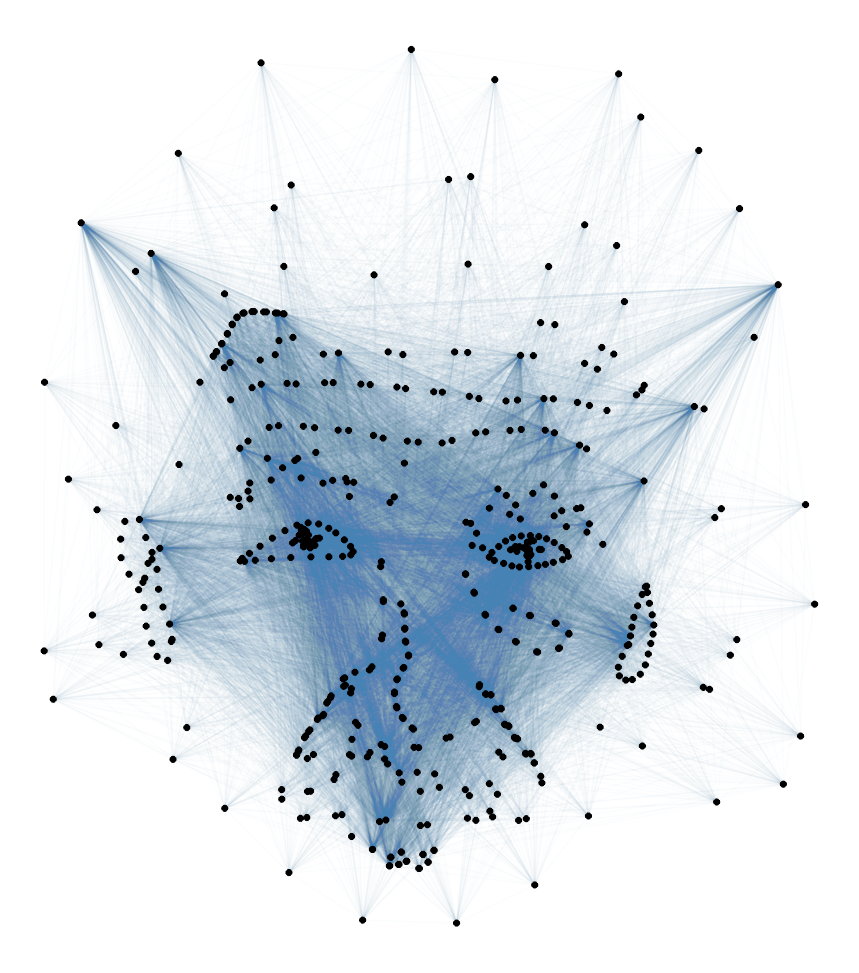}
		\label{fig:pi-net1}
}
	\caption {
		Greek alphabet $\pi$ pattern and networks
	}
	\label{fig_ein0}
\end{figure}
 
As in the previous subsection, we perform the numerical simulations for the time evolution of solutions, shown in Fig.~\ref{fig_ein2}, that forms the Einstein's face image in this setting.

\begin{figure}[h]
	\centering
	\subfloat[t=0.000]{
		\includegraphics[scale=0.7]{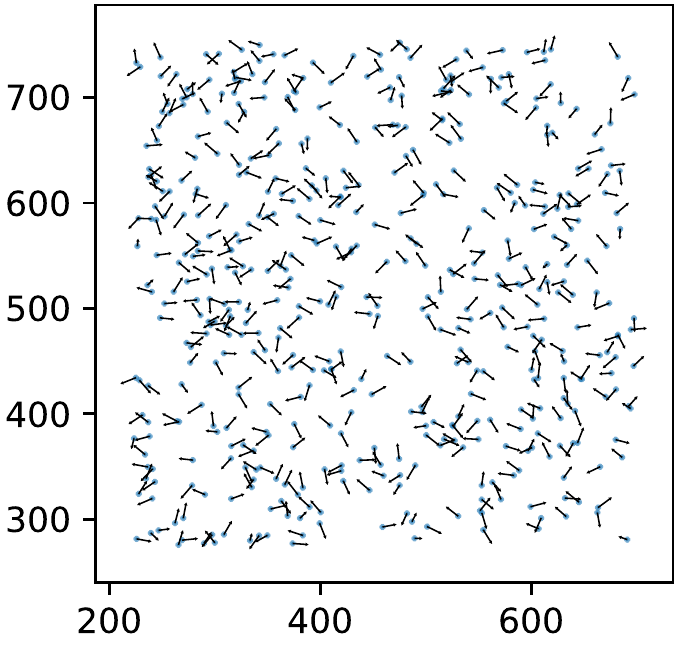}
		\label{fig:ein1}
	}
	\subfloat[t=0.125]{
		\includegraphics[scale=0.7]{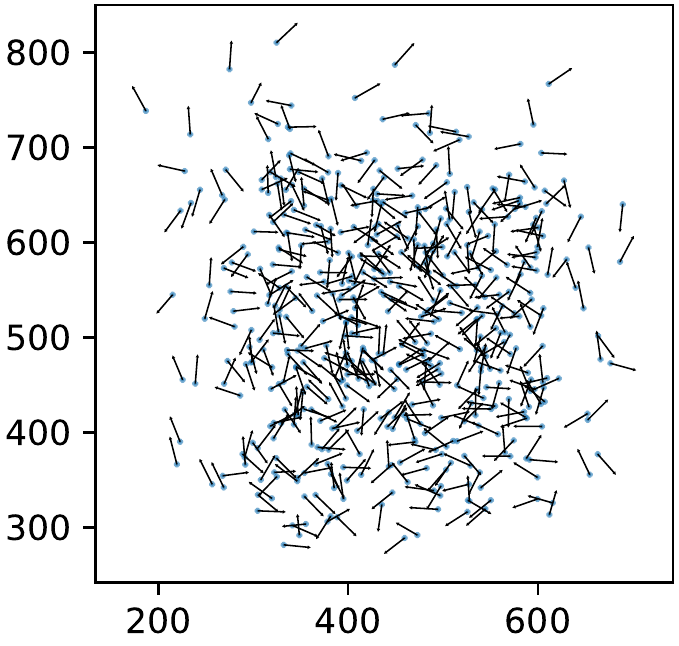}
		\label{fig:ein2}
	}
	\subfloat[t=0.513]{
		\includegraphics[scale=0.7]{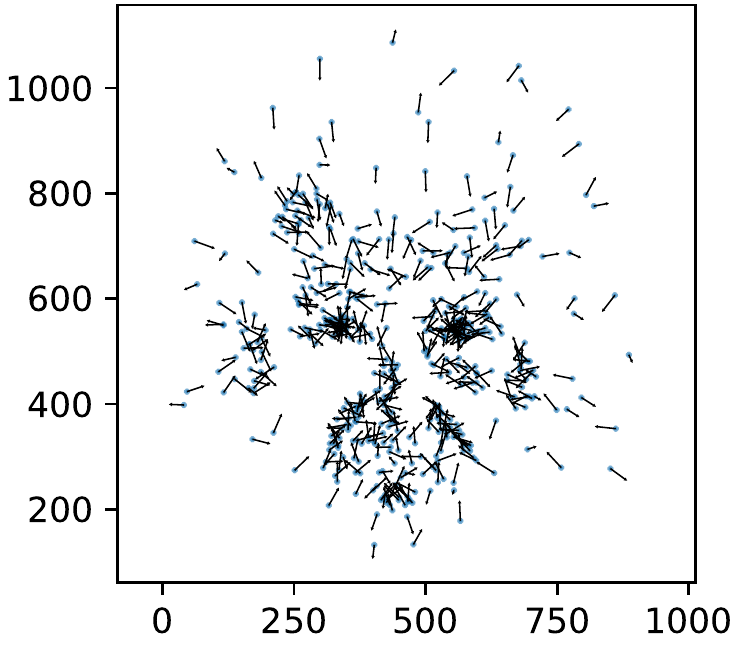}
		\label{fig:ein3}
	}
	
	\subfloat[t=0.913]{
		\includegraphics[scale=0.7]{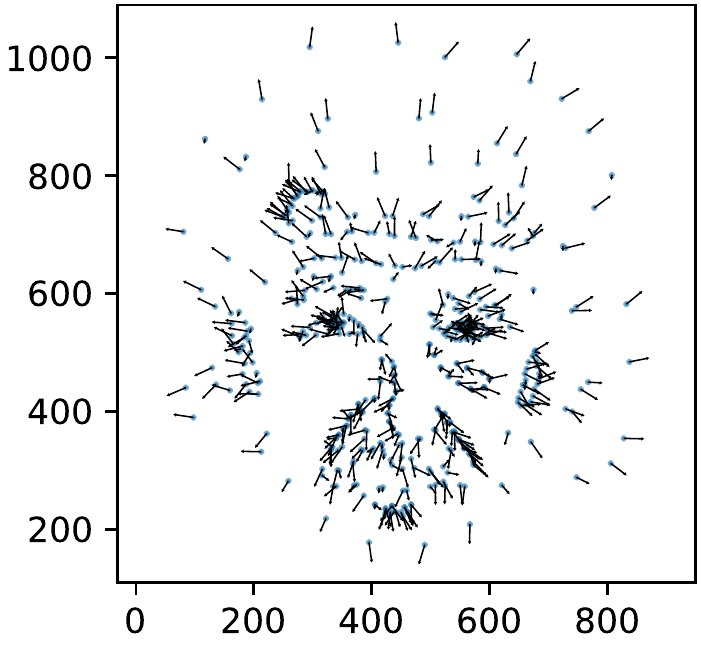}
		\label{fig:ein4}
	}
	\subfloat[t=1.550]{
		\includegraphics[scale=0.7]{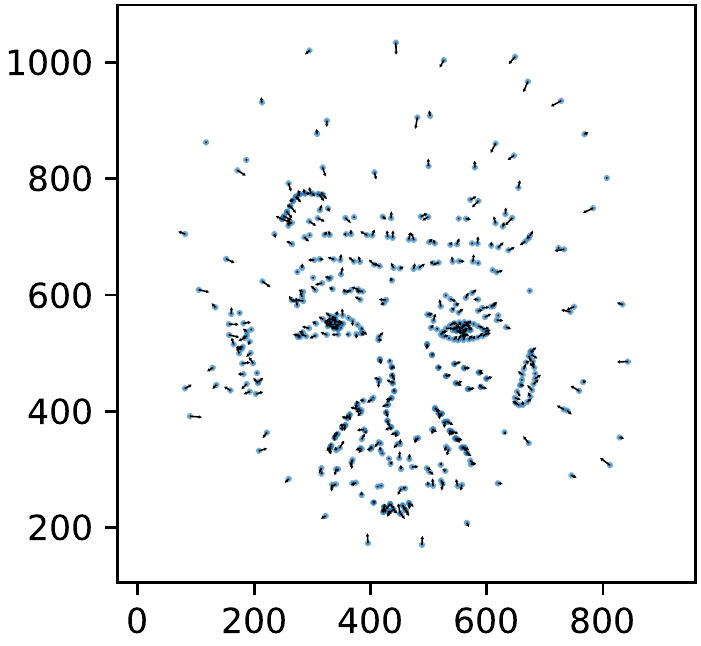}
		\label{fig:ein5}
	}
	\subfloat[t=2.238]{
		\includegraphics[scale=0.7]{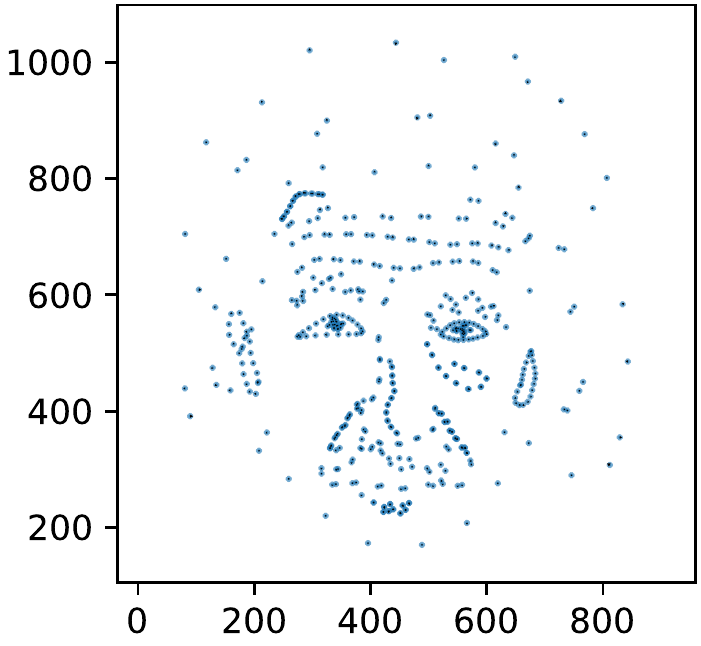}
		\label{fig:ein6}
	}

	\caption{
		Snapshots for an Einstein's face image pattern simulation 
	}
	\label{fig_ein2}
\end{figure}

%%%%%%%%%%%%%%%%%%%%%%%%%%%%%%%%%%%%%%%%%%%%%%%%%%
%
%
%
% \section*{Acknowledgement}
%
%
%%%%%%%%%%%%%%%%%%%%%%%%%%%%%%%%%%%%%%%%%%%%%%%%%%

\section*{Acknowledgement}  
YPC acknowledges support from NRF grant (No. 2017R1C1B2012918), POSCO Science Fellowship of POSCO TJ Park Foundation, and Yonsei University Research Fund of 2019-22-0212 and 2020-22-0505. OT acknowledges support from NWO Vidi grant 016.Vidi.189.102, ``Dynamical-Variational Transport Costs and Application to Variational Evolutions".
% ----------------------------------------------------------------
\bibliographystyle{amsplain}
\bibliography{SCS_network}

\providecommand{\bysame}{\leavevmode\hbox to3em{\hrulefill}\thinspace}
\providecommand{\MR}{\relax\ifhmode\unskip\space\fi MR }
% \MRhref is called by the amsart/book/proc definition of \MR.
\providecommand{\MRhref}[2]{%
  \href{http://www.ams.org/mathscinet-getitem?mr=#1}{#2}
}
\providecommand{\href}[2]{#2}
\begin{thebibliography}{10}

\bibitem{ACHL12}
Shin~Mi Ahn, Heesun Choi, Seung-Yeal Ha, and Ho~Lee, \emph{On
  collision-avoiding initial configurations to {C}ucker-{S}male type flocking
  models}, Commun. Math. Sci. \textbf{10} (2012), no.~2, 625--643. \MR{2901323}

\bibitem{AH10}
Shin~Mi Ahn and Seung-Yeal Ha, \emph{Stochastic flocking dynamics of the
  {C}ucker-{S}male model with multiplicative white noises}, J. Math. Phys.
  \textbf{51} (2010), no.~10, 103301, 17. \MR{2761313}

\bibitem{BCC11}
Fran\c{c}ois Bolley, Jos\'{e}~A. Ca\~{n}izo, and Jos\'{e}~A. Carrillo,
  \emph{Stochastic mean-field limit: non-{L}ipschitz forces and swarming},
  Math. Models Methods Appl. Sci. \textbf{21} (2011), no.~11, 2179--2210.
  \MR{2860672}

\bibitem{BB68}
John Buck and Elisabeth Buck, \emph{Mechanism of rhythmic synchronous flashing
  of fireflies}, Science \textbf{159} (1968), no.~3821, 1319--1327.

\bibitem{CCMP17}
Jos\'{e}~A. Carrillo, Young-Pil Choi, Piotr~B. Mucha, and Jan Peszek,
  \emph{Sharp conditions to avoid collisions in singular {C}ucker-{S}male
  interactions}, Nonlinear Anal. Real World Appl. \textbf{37} (2017), 317--328.
  \MR{3648384}

\bibitem{CCP17}
Jos\'{e}~A. Carrillo, Young-Pil Choi, and Sergio~P. Perez, \emph{A review on
  attractive-repulsive hydrodynamics for consensus in collective behavior},
  Active particles. {V}ol. 1. {A}dvances in theory, models, and applications,
  Model. Simul. Sci. Eng. Technol., Birkh\"{a}user/Springer, Cham, 2017,
  pp.~259--298. \MR{3644593}

\bibitem{CFRT10}
Jos\'{e}~A. Carrillo, Massimo Fornasier, Jes\'{u}s Rosado, and Giuseppe
  Toscani, \emph{Asymptotic flocking dynamics for the kinetic {C}ucker-{S}male
  model}, SIAM J. Math. Anal. \textbf{42} (2010), no.~1, 218--236. \MR{2596552}

\bibitem{CFTV10}
Jos\'{e}~A. Carrillo, Massimo Fornasier, Giuseppe Toscani, and Francesco Vecil,
  \emph{Particle, kinetic, and hydrodynamic models of swarming}, Mathematical
  modeling of collective behavior in socio-economic and life sciences, Model.
  Simul. Sci. Eng. Technol., Birkh\"{a}user Boston, Boston, MA, 2010,
  pp.~297--336. \MR{2744704}

\bibitem{CDP18}
Patrick Cattiaux, Fanny Delebecque, and Laure P\'{e}d\`eches, \emph{Stochastic
  {C}ucker-{S}male models: old and new}, Ann. Appl. Probab. \textbf{28} (2018),
  no.~5, 3239--3286. \MR{3847987}

\bibitem{CHL17}
Young-Pil Choi, Seung-Yeal Ha, and Zhuchun Li, \emph{Emergent dynamics of the
  {C}ucker-{S}male flocking model and its variants}, Active particles. {V}ol.
  1. {A}dvances in theory, models, and applications, Model. Simul. Sci. Eng.
  Technol., Birkh\"{a}user/Springer, Cham, 2017, pp.~299--331. \MR{3644594}

\bibitem{CH17}
Young-Pil Choi and Jan Haskovec, \emph{Cucker-{S}male model with normalized
  communication weights and time delay}, Kinet. Relat. Models \textbf{10}
  (2017), no.~4, 1011--1033. \MR{3622098}

\bibitem{CKPP19}
Young-Pil Choi, Dante Kalise, Jan Peszek, and Andr\'{e}s~A. Peters, \emph{A
  collisionless singular {C}ucker-{S}male model with decentralized formation
  control}, SIAM J. Appl. Dyn. Syst. \textbf{18} (2019), no.~4, 1954--1981.
  \MR{4028780}

\bibitem{CL18}
Young-Pil Choi and Zhuchun Li, \emph{Emergent behavior of {C}ucker-{S}male
  flocking particles with heterogeneous time delays}, Appl. Math. Lett.
  \textbf{86} (2018), 49--56. \MR{3836802}

\bibitem{CLHXY14}
Young-Pil Choi, Zhuchun Li, Seung-Yeal Ha, Xiaoping Xue, and Seok-Bae Yun,
  \emph{Complete entrainment of {K}uramoto oscillators with inertia on networks
  via gradient-like flow}, J. Differential Equations \textbf{257} (2014),
  no.~7, 2591--2621. \MR{3228978}

\bibitem{CP19}
Young-Pil Choi and Cristina Pignotti, \emph{Emergent behavior of
  {C}ucker-{S}male model with normalized weights and distributed time delays},
  Netw. Heterog. Media \textbf{14} (2019), no.~4, 789--804. \MR{4026010}

\bibitem{CS19}
Young-Pil Choi and Samir Salem, \emph{Cucker-{S}male flocking particles with
  multiplicative noises: stochastic mean-field limit and phase transition},
  Kinet. Relat. Models \textbf{12} (2019), no.~3, 573--592. \MR{3928122}

\bibitem{CS07}
Felipe {Cucker} and Steve {Smale}, \emph{Emergent behavior in flocks}, IEEE
  Transactions on Automatic Control \textbf{52} (2007), no.~5, 852--862.

\bibitem{DHK19}
Jiu-Gang Dong, Seung-Yeal Ha, and Doheon Kim, \emph{Interplay of time-delay and
  velocity alignment in the {C}ucker-{S}male model on a general digraph},
  Discrete Contin. Dyn. Syst. Ser. B \textbf{24} (2019), no.~10, 5569--5596.
  \MR{4026940}

\bibitem{D96}
Richard Durrett, \emph{Stochastic calculus}, Probability and Stochastics
  Series, CRC Press, Boca Raton, FL, 1996, A practical introduction.
  \MR{1398879}

\bibitem{EHS16}
Radek Erban, Jan Ha\v{s}kovec, and Yongzheng Sun, \emph{A {C}ucker-{S}male
  model with noise and delay}, SIAM J. Appl. Math. \textbf{76} (2016), no.~4,
  1535--1557. \MR{3534479}

\bibitem{HJNXZ17}
Seung-Yeal Ha, Jiin Jeong, Se~Eun Noh, Qinghua Xiao, and Xiongtao Zhang,
  \emph{Emergent dynamics of {C}ucker-{S}male flocking particles in a random
  environment}, J. Differential Equations \textbf{262} (2017), no.~3,
  2554--2591. \MR{3582237}

\bibitem{HLL09}
Seung-Yeal Ha, Kiseop Lee, and Doron Levy, \emph{Emergence of time-asymptotic
  flocking in a stochastic {C}ucker-{S}male system}, Commun. Math. Sci.
  \textbf{7} (2009), no.~2, 453--469. \MR{2536447}

\bibitem{HL09}
Seung-Yeal Ha and Jian-Guo Liu, \emph{A simple proof of the {C}ucker-{S}male
  flocking dynamics and mean-field limit}, Commun. Math. Sci. \textbf{7}
  (2009), no.~2, 297--325. \MR{2536440}

\bibitem{HT08}
Seung-Yeal Ha and Eitan Tadmor, \emph{From particle to kinetic and hydrodynamic
  descriptions of flocking}, Kinet. Relat. Models \textbf{1} (2008), no.~3,
  415--435. \MR{2425606}

\bibitem{HM20}
Jan Haskovec and Ioannis Markou, \emph{Asymptotic flocking in the
  {C}ucker-{S}male model with reaction-type delays in the non-oscillatory
  regime}, Kinet. Relat. Models \textbf{13} (2020), no.~4, 795--813.
  \MR{4112181}

\bibitem{MMPZ19}
Piotr Minakowski, Piotr~B. Mucha, Jan Peszek, and Ewelina Zatorska,
  \emph{Singular {C}ucker-{S}male dynamics}, Active particles, {V}ol. 2, Model.
  Simul. Sci. Eng. Technol., Birkh\"{a}user/Springer, Cham, 2019, pp.~201--243.
  \MR{3932462}

\bibitem{MT11}
Sebastien Motsch and Eitan Tadmor, \emph{A new model for self-organized
  dynamics and its flocking behavior}, J. Stat. Phys. \textbf{144} (2011),
  no.~5, 923--947. \MR{2836613}

\bibitem{Pes14}
Jan Peszek, \emph{Existence of piecewise weak solutions of a discrete
  {C}ucker-{S}male's flocking model with a singular communication weight}, J.
  Differential Equations \textbf{257} (2014), no.~8, 2900--2925. \MR{3249275}

\bibitem{R12}
A.~J. Roberts, \emph{Modify the improved euler scheme to integrate stochastic
  differential equations}, preprint, arXiv:1210.0933.

\bibitem{Sum10}
David J.~T. Sumpter, \emph{Collective animal behavior}, Princeton University
  Press, 2010.

\bibitem{TLY14}
Ta~Viet Ton, Nguyen Thi~Hoai Linh, and Atsushi Yagi, \emph{Flocking and
  non-flocking behavior in a stochastic {C}ucker-{S}male system}, Anal. Appl.
  (Singap.) \textbf{12} (2014), no.~1, 63--73. \MR{3150970}

\bibitem{VZ12}
Tam\'as Vicsek and Anna Zafeiris, \emph{Collective motion}, Physics Reports
  \textbf{517} (2012), no.~3, 71 -- 140, Collective motion.

\bibitem{War49}
J.~B. {Ward}, \emph{Equivalent circuits for power-flow studies}, Electrical
  Engineering \textbf{68} (1949), no.~9, 794--794.

\end{thebibliography}
% ----------------------------------------------------------------

\end{document}